\documentclass[11pt,aps,prb, preprint,reqno]{amsart}
\usepackage[margin=0.9in]{geometry}
\usepackage{amsmath,amssymb,amsthm,graphicx,amsxtra, setspace}
\usepackage{mathabx}
\usepackage[colorlinks,backref]{hyperref}
\usepackage{esint}
\usepackage{hyperref}
\usepackage[dvipsnames]{xcolor}
\usepackage{dsfont}
\usepackage[utf8]{inputenc}
\usepackage{mathrsfs}
\usepackage{upgreek}
\usepackage{mathtools}
\allowdisplaybreaks

\usepackage[cyr]{aeguill}

\colorlet{darkblue}{blue!50!black}

\hypersetup{
	colorlinks,%
	citecolor=blue,%
	filecolor=red,%
	linkcolor=darkblue,%
	urlcolor=blue,%
	pdfnewwindow=true,%
	pdfstartview={FitH}
}

\colorlet{darkblue}{red!100!black}

\newtheorem{theorem}{Theorem}[section]
\newtheorem{proposition}[theorem]{Proposition}
\theoremstyle{plain}
\newtheorem{lemma}[theorem]{Lemma}

\theoremstyle{definition}
\newtheorem{definition}{Definition}[section]

\newtheorem{remark}{Remark}[section]

\newtheorem*{maintheorem*}{Main Theorem}
\newtheorem*{maincorollary*}{Main Corollary}

\newtheorem{hypothesis}[theorem]{Hypothesis}

\newcommand{\norm}[1]{\left\|#1\right\|}
\newcommand{\abs}[1]{\left|#1\right|}

\newcommand{\R}{\ensuremath{\mathbb{R}}}
\newcommand{\N}{\ensuremath{\mathbb{N}}}

\newcommand{\supp}{\ensuremath{\mathrm{supp}\,}}

\newcommand{\ip}[2]{\fourIdx{}{0}{}{\!x}{\mathcal{ X}}}

\def\N{\mathbb{N}}
\def\V{V}

\def\u{\boldsymbol{u}}
\def\v{\boldsymbol{v}}
\def\w{\boldsymbol{w}}
\def\f{\boldsymbol{f}}

\def\h{\boldsymbol{h}}
\def\n{\boldsymbol{n}}
\def\boldphi{\boldsymbol{\phi}}
\def\D{\mathrm{D}}
\def\V{\mathbb{V}}
\def\L{\mathbb{L}}
\def\H{\mathbb{H}}
\def\P{\mathrm{P}}

\def\A{\mathcal{A}}

\def\Lrm{\mathrm{L}}
\def\B{\mathcal{B}}

\def\d{\mathrm{d}}
\def\C{\mathrm{C}}
\def\wi{\widetilde}

\def\mfrak{\mathfrak{m}}

\def\vfrak{\mathfrak{v}}

\def\pfrak{\mathfrak{p}}
\def\qfrak{\mathfrak{q}}
\def\rfrak{\mathfrak{r}}
\def\zfrak{\mathfrak{z}}

\numberwithin{equation}{section} \allowdisplaybreaks

\title[Optimal control problem for 3D critical CBFEs]{\small Optimal control problem associated with three-dimensional critical convective Brinkman-Forchheimer equations}


\author[Kush Kinra]{Kush Kinra}
\address[Kush Kinra]{Center for Mathematics and Applications (NOVA Math), NOVA School of Science and Technology (NOVA FCT),	Portugal.}
\email[Kush Kinra]{kushkinra@gmail.com, k.kinra@fct.unl.pt}

\author[Fernanda Cipriano]{Fernanda Cipriano}
\address[Fernanda Cipriano]{
	Center for Mathematics and Applications (NOVA Math) and Department of Mathematics, NOVA School of Science and Technology (NOVA FCT), Portugal.}
\email[Fernanda Cipriano]{cipriano@fct.unl.pt}

\allowdisplaybreaks[1]

\usepackage{comment}

\begin{document}

\begin{abstract}
	In this article, we are concerned about the velocity tracking optimal control problem for 3D critical convective Brinkman–Forchheimer equations defined on a  simply connected bounded domain $\mathfrak{D} \subset \mathbb{R}^3$ with $\mathrm{C}^2$-boundary $\partial\mathfrak{D}$. The control is introduced through an external force. The objective is to optimally minimize a velocity tracking cost functional, for which the velocity vector field is oriented towards a target velocity.  Most importantly, we are concerned about the first-order necessary optimality conditions for above-mentioned optimal control problem which is the main challenging task of this article. To overcome the difficulties related to the differentiability of the control-to-state mapping, consequence of the lack of regularity of the state variable on bounded domains, we first establish some intermediate optimality conditions and then pass to the limit.
\end{abstract}

\maketitle


\section{Introduction}
\setcounter{equation}{0}

This article discusses the velocity tracking optimal control problem associated with 3D  critical convective Brinkman–Forchheimer equations (CBFEs) defined on a  simply connected bounded domain $\mathfrak{D} \subset \mathbb{R}^3$ with $\mathrm{C}^2$-boundary $\partial\mathfrak{D}$. The critical CBFEs (also known as damped Navier-Stokes equations (NSEs)) describe the motion of incompressible viscous fluid through a rigid, homogeneous, isotropic, porous medium which are read as (\cite{Hajduk+Robinson}):
\begin{equation}\label{1}
	\left\{
	\begin{aligned}
		\frac{\partial \mfrak}{\partial t}-\mu \Delta\mfrak+(\mfrak\cdot\nabla)\mfrak+\alpha\mfrak+\beta|\mfrak|^{2}\mfrak+\nabla \pi&=\boldsymbol{f}, && \text{ in } \ \mathfrak{D}\times(0,\infty), \\ \nabla\cdot\mfrak&=0, && \text{ in } \ \mathfrak{D}\times[0,\infty), \\
		\mfrak&=\boldsymbol{0}, &&  \text{ on } \ \partial\mathfrak{D}\times[0,\infty), \\
		\mfrak(0)&=\mfrak_0, && \text{ in } \ \mathfrak{D},\\
		\int_{\mathfrak{D}}\pi(x,t)\d x&=0, &&  \text{ in } \ [0,\infty),
	\end{aligned}
	\right.
\end{equation}
 where $\mfrak$, $\pi$ and $\f$  denote the velocity field, the scalar pressure field,  and an external forcing, respectively. The last condition in \eqref{1} is there to make sure that the pressure $\pi$ is unique. Here, the constants $\mu, \alpha, \beta >0$ denote the Brinkman (effective viscosity),   Darcy (permeability of porous medium) and   Forchheimer (proportional to the porosity of the material) coefficients, respectively. In context of 3D NSEs, Leray (\cite{Leray_1934}) and Hopf (\cite{Hopf_1951}) were the first who established the existence of at least one weak solution, known as \emph{Leray-Hopf weak solutions}. But the uniqueness of Leray-Hopf weak solutions is still an open problem. Whereas, it has been proved in \cite{Gautam+Mohan_2025} that the system \eqref{1} admits a unique Leray-Hopf weak solution (in the sense of Definition \ref{def-weak-sol} below).    It has also been proved by Hajduk and Robinson \cite[Proposition 1.1]{Hajduk+Robinson} that system \eqref{1} and the NSEs have the same scaling when $\alpha=0$, but have no scaling invariance for other values of $\alpha$.


In the works \cite{MTM_AMOP_2021,MTM_EECT_2022,MTM_Optimization_2022}, the author considered two-dimensional CBFEs on bounded domains and investigated the existence of optimal controls with first-order necessary optimality conditions. However, authors in \cite{Gautam+Mohan_Arxiv_2025} considered 2D as well as 3D CBFEs on a torus and used dynamic programming to look at the infinite-dimensional Hamilton-Jacobi-Bellman equation  of first order that is connected to an optimal control problem for underlying system. It is worth noting that the work \cite{Gautam+Mohan_Arxiv_2025} requires the smooth solutions of state equation which is only known on a torus. In this work, we aim to investigate the optimal control problem associated with 3D critical CBFEs defined on bounded domains and we also establish first-order necessary optimality conditions. This control problem has numerous industrial and engineering applications (see \cite{Gunzburger_1995,Sritharan_1998}).

Let us now discuss the advantages gained in two-dimensional CBFEs and the challenges that arise when extending the analysis to three-dimensional CBFEs \eqref{1}. Let $\f$ and $\boldsymbol{g}$ be in $\Lrm^2(0,T;\L^2(\mathfrak{D}))$. For $0<\rho<1$, let $\f_{\rho} = \f + \rho \boldsymbol{g}$, $(\mfrak_{\f},\pi_{\f})$ and $(\mfrak_\rho,\pi_{\rho})$ be two weak solutions of \eqref{1} corresponding to $\f$ and $\f_{\rho}$, and define $\zfrak_\rho: = \frac{\mfrak_{\rho} - \mfrak_{\f}}{\rho}$. To establish the existence of an optimal control, there is no significant difference between the two-dimensional and three-dimensional cases; however, the main difficulty arises when deriving the first-order necessary optimality conditions (F-NOC). To deal with F-NOC, we usually start by finding Lipschitz estimates (see Proposition \ref{prop-stability-diff} below). These estimates help us to find uniform estimates for $\{\zfrak_\rho\}_{\rho>0}$ in a suitable functional setting and show that the limit $\zfrak$ is the solution of a corresponding linearized system. The G\^ateaux differentiability of the control-to-state mapping is intimately linked to whether or not this linearized system can be solved, as well as the uniqueness and regularity of its solutions. Then we introduce the adjoint system corresponding to this linearized system, and derive the F-NOC. In this approach, we normally need to make sure that both systems are well-posed and have unique solution. In addition, we also require that the solution of adjoint equation is an admissible test function for the weak formulation of the linearized system, and the solution of linearized equation is an admissible test function for the weak formulation of the adjoint system.

It is worth mentioning here that for system \eqref{1} in 2D (see \cite{MTM_AMOP_2021,MTM_EECT_2022,MTM_Optimization_2022}), the solution of linearized and adjoint systems belongs to $\mathrm{C}([0,T];\L^2(\mathfrak{D}))\cap \Lrm^2(0,T;\H^1_0(\mathfrak{D}))$ (for any given $T>0$) and its weak time derivative belongs to $\Lrm^2(0,T;\H^{-1}(\mathfrak{D}))+\Lrm^{\frac{4}{3}}(0,T;\L^{\frac43}(\mathfrak{D}))$. Due to Ladyzhenskaya's inequality in 2D (see \eqref{eqn-lady} for $d=2$ below), the space $\mathrm{C}([0,T];\L^2(\mathfrak{D}))\cap \Lrm^2(0,T;\H^1_0(\mathfrak{D}))$ is continuously  embedded in $\Lrm^{4}(0,T;\L^{4}(\mathfrak{D}))$, and therefore the solution of adjoint equation is an admissible test function for the weak formulation of the linearized system, and the solution of linearized equation is an admissible test function for the weak formulation of the adjoint system. However, in 3D, solutions do not belong to $\Lrm^{4}(0,T;\L^{4}(\mathfrak{D}))$. Due to insufficient regularity, the solutions of the linearized and adjoint equations cannot be employed as test functions in the weak formulations of one another. This tells us that no direct relation can be established between the solutions of linearized and adjoint equations, and under these circumstances, deriving the F-NOC becomes a highly nontrivial task.

In order to overcome the difficulty discussed above in 3D, we follow the approach explored in the work \cite{Arada_2014}.  We derive intermediate optimality conditions by utilizing the equation for $\zfrak_{\rho}$ (see \eqref{eqn-difference} below) and its corresponding intermediate adjoint equation (see \eqref{eqn-adjoint} below), rather than the usual linearized equation. In this approach, we require that $\zfrak_{\rho}$ possesses the same regularity as the state variable. To get the regularity of $\zfrak_{\rho}$ same as state variable, we implement the following hypothesis on the coefficients $\mu$ and $\beta$ appearing in system \eqref{1}:
\begin{hypothesis}\label{Para-Hypo}
	The parameters $\mu$ and $\beta$ satisfy:
	\begin{align*}
		2\beta \mu > \frac{1}{\kappa}, \;\; \text{ for some }\; 0<\kappa <1.
	\end{align*}
\end{hypothesis}
On the other hand, challenges concerning the existence, uniqueness, and regularity of a weak solution $\qfrak_\rho$ to the associated intermediate adjoint equation \eqref{eqn-adjoint} persist. To address this difficulty, we regularize the intermediate adjoint equation by introducing a suitable nonlinear term (see system \eqref{eqn-adjoint-regular} below). We then prove existence and uniqueness of the corresponding regularized intermediate adjoint variable $\qfrak^{\delta}$, establish a relation between $\zfrak_{\rho}$ and $\qfrak^{\delta}$, and finally pass to the limit as the regularization parameter $\delta\to 0$.

\subsection*{Main result}
Our main goal is to control the solution of the system \eqref{1} by a distributed force $\f$. The control variable $\f$ belongs to the set $\mathcal{F}_{ad}$ of admissible controls, which
is defined as a nonempty bounded closed convex subset of $\mathrm{L}^2(0,{T};{\L}^{2}(\mathfrak{D})))$.

Let $\mfrak$ be the solution to system \eqref{1} corresponding to control $\f\in \mathcal{F}_{ad}$. We consider the cost functional $ \mathcal{J}: \mathcal{F}_{ad}  \to \mathbb{R}^+$ given by 
\begin{align}\label{eqn-cost-functional}
	\mathcal{J}(\f,\mfrak  )=\frac12  \int_0^{T} \|\mfrak(t)-\mfrak_d(t)\|^2_2 \d t + \frac{\lambda}{2} \int_0^{T}\|\f(t)\|_2^2\d t,
\end{align}
where $\mfrak_d \in \mathrm{L}^2(0,{T};{\L}^2(\mathfrak{D}))$ corresponds to a desired target field and any $\lambda> 0$. The control problem reads
\begin{align}\label{eqn-control-problem}
	\min_{\f\in\mathcal{F}_{ad}}\bigg\{\mathcal{J}(\f,\mfrak) \; : \; \mfrak \text{ is the solution of \eqref{1} with force } \f \bigg\}.
\end{align}

Our main result demonstrates the existence of a solution to the control problem and formulates the first-order necessary optimality conditions.
\begin{theorem}\label{main-thm1}
	Assume that $\mfrak_0 \in \L^2(\mathfrak{D})$. Then the control problem \eqref{eqn-control-problem} admits, at least, one optimal solution 
	\begin{align}
		(\widetilde{\f}, \widetilde{\mfrak}) \in \mathcal{F}_{ad}\times  \mathrm{C}([0,{T}]; \L^2(\mathfrak{D}))\cap\mathrm{L}^2(0,{T}; \H^1_0(\mathfrak{D}))\cap\mathrm{L}^4(0,{T}; \L^4(\mathfrak{D})) ,
	\end{align}
	where $\widetilde{\mfrak}$ is the unique solution of \eqref{1} with $\f$ replaced by $\widetilde{\f}$.  In addition, let the Hypothesis \ref{Para-Hypo} be satisfied, then there exists a weak solution $\wi\qfrak $ (in the sense of Definition \ref{def-main-adjoint} below)  of the following system:
	\begin{equation}
		\left\{
		\begin{aligned}
			-\frac{\d\wi\qfrak}{\d t} + \mu \A \wi\qfrak  - \B(\wi\mfrak , \wi\qfrak) + \mathcal{P}[ \sum_{j=1}^3 [\nabla (\wi\mfrak)_j]\wi\qfrak_j ]  + \alpha\wi\qfrak       +  \beta  \mathcal{P}\{ |\wi\mfrak|^2 \wi\qfrak \}     +  2 \beta \mathcal{P}\{ [\wi\mfrak \cdot\wi\qfrak ] \wi\mfrak \}  & = \mathcal{P} (\wi \mfrak - \mfrak_d), \\
			\wi\qfrak(T)&=\boldsymbol{0}.
		\end{aligned}
		\right.
	\end{equation} 
	such that 
	\begin{align}
		\int_{0}^{T} (\v(t) - \wi\f(t)   ,  \wi\qfrak (t) + \lambda \wi\f(t))\d t \geq 0, \;\; \; \text{ for all }\;\; \v \in  \mathcal{F}_{ad}.
	\end{align}
\end{theorem}

\subsection*{Organization of the paper}

In next section, we introduce the necessary functional framework, including the relevant function spaces and the definitions and properties of both linear and nonlinear operators that will be used to obtain an abstract formulation. In Section \ref{sec3}, we present the abstract formulation of the state equation, define the concept of a weak solution, and discuss the well-posedness results along with Lipschitz continuity estimates. Section \ref{sec4} is devoted to the intermediate adjoint system, where we prove the existence of a weak solution by regularizing the original system. Additionally, we derive an intermediate duality relation that plays a key role in the optimality analysis. Finally, in Section \ref{sec5}, we establish our main result concerning the existence of an optimal control-state pair for the control problem, and we derive the corresponding first-order necessary optimality conditions.

\section{Mathematical Formulation}\label{sec2}\setcounter{equation}{0}

We cover a few functional spaces, linear operators, and nonlinear operators in this section to facilitate future discussions. For a detailed discussion on the content provided in this section, we refer readers to \cite[Section 2]{MTM_SPDE_2022}.

\subsection{Functional setting}
Let us define  $\mathcal{V}:=\left\{\mfrak\in\C_0^{\infty}(\mathfrak{D};\R^3):\nabla\cdot\mfrak=0\right\},$  where  $\C_0^{\infty}(\mathfrak{D};\R^3)$ is the space  of all infinitely differentiable functions  with compact support in $\mathfrak{D}$.  We define the spaces  $\H$, $\V$ and $\wi\L^p$  as the completion of $\mathcal{V}$ in $\L^2(\mathfrak{D})=\mathrm{L}^2(\mathfrak{D};\R^3)$, $\H^1(\mathfrak{D})=\mathrm{H}^1(\mathfrak{D};\R^3)$ and $\L^p(\mathfrak{D})=\mathrm{L}^p(\mathfrak{D};\R^3)$ norms, respectively. The spaces $\H$, $\V$ and $\wi\L^p$ can be characterized as 
\begin{align*}
	\H & :=\left\{\mfrak\in\L^2(\mathfrak{D}):\mfrak\cdot\n\big|_{\partial\mathfrak{D}}=0, \;\; \nabla\cdot\mfrak=0 \right\},\\
	\V & :=\left\{\mfrak\in\H^1_0(\mathfrak{D}):\nabla\cdot\mfrak=0\right\}\\
	\wi\L^p & :=\left\{\mfrak\in\L^p(\mathfrak{D}):\nabla\cdot\mfrak=0,\mfrak\cdot\n\big|_{\partial\mathfrak{D}}=0\right\}, \;\;\; \text{ for  } \;\; p\in(2,\infty),
\end{align*}
with   the norms   $\|\mfrak\|_{2}^2:=\int_{\mathfrak{D}}|\mfrak(x)|^2\d x$, $\|\mfrak\|_{\V}^2:= \int_{\mathfrak{D}}|\mfrak(x)|^2\d x+\int_{\mathfrak{D}}|\nabla\mfrak(x)|^2\d x$ and $\|\mfrak\|_{p}^p := \int_{\mathfrak{D}}|\mfrak(x)|^p\d x,$ respectively, where $\n$ is the outward normal to the boundary $\partial\mathfrak{D}$  and $\mfrak\cdot\n\big|_{\partial\mathfrak{D}}$ should be understood in the sense of trace in $\H^{-1/2}(\partial\mathfrak{D})$. Using the Poincar\'e inequality (see \eqref{2.7a} below),   this norm is equivalent to the norm $ \|\mfrak\|_{\V}^2:=\int_{\mathfrak{D}}|\nabla\mfrak(x)|^2\d x.$

The inner product in the Hilbert space $\H$ is  denoted   by $(\cdot,\cdot)$  and the induced duality, for instance between the spaces $\V$  and its dual $\V'$, and $\wi\L^p$ and its dual $\wi\L^{\frac{p}{p-1}}$ is expressed by $\left<\cdot,\cdot\right>$.     $\V$ is densely and continuously embedded into $\H$ and $\H$ can be identified with its dual $\H'$ and we have the  \textit{Gelfand triple}: $\V\subset	\H\equiv\H'\subset\V'.$     In addition, the embedding of $\V\subset\H$ is compact.  

\subsection{Linear operator}
 It is well-known  that every vector field $\mfrak\in\mathbb{L}^p(\mathfrak{D})$, for $1<p<\infty$ can be uniquely represented as $\mfrak=\qfrak+\nabla \varphi,$ where $\qfrak\in\mathbb{L}^p(\mathfrak{D})$ with $\mathrm{div \ }\qfrak = 0$ in the sense of distributions in $\mathfrak{D}$, $\qfrak\cdot\n=0$ on $\partial\mathfrak{D}$ and $\varphi \in\mathrm{W}^{1,p}(\mathfrak{D})$ (Helmholtz-Weyl or Helmholtz-Hodge decomposition).  For smooth vector fields in $\mathfrak{D}$, such a decomposition is an orthogonal sum in $\mathbb{L}^2(\mathfrak{D})$.     $\mfrak=\qfrak+\nabla \varphi$ holds for all $\mfrak\in\mathbb{L}^p(\mathfrak{D})$, so that we can define the projection operator $\mathcal{P}_p$  by $\mathcal{P}_p\mfrak = \qfrak$. For simplicity, we use the notation $\mathcal{P}$ instead of $\mathcal{P}_2$ in the rest of the paper.  Since $\mathfrak{D}$ is of $\mathrm{C}^2$-boundary, we also infer that $\mathcal{P}$ maps $\H^1(\mathfrak{D})$ into itself and is continuous for the norm of $\H^1(\mathfrak{D})$. Let us define
\begin{equation*}
	\A\mfrak:=-\mathcal{P}\Delta\mfrak,\ \mfrak\in\D(\A):=\V\cap\H^{2}(\mathfrak{D}).
\end{equation*}
    the operator $\A$ is a non-negative, self-adjoint operator in $\H$  and \begin{align}\label{2.7a}
    	\left<\A\mfrak,\mfrak\right>=\|\mfrak\|_{\V}^2,\ \textrm{ for all }\ \mfrak\in\V, \ \text{ so that }\ \|\A\mfrak\|_{\V'}\leq \|\mfrak\|_{\V}.
    \end{align}


\begin{remark}\label{Rem-2.1}
	It is well known  that for a bounded domain $\mathfrak{D}$, the operator $\A$ is invertible and its inverse $\A^{-1}$ is bounded, self-adjoint and compact in $\H$. Hence the spectrum of $\A$ consists of an infinite sequence $0< \lambda_1\leq \lambda_2\leq\ldots\leq \lambda_k\leq \ldots$ 
    of eigenvalues with $\lambda_k\to\infty$ as $k\to\infty$, and there exists an orthogonal basis $\{e_k\}_{k=1}^{\infty} $ of $\H$ consisting of eigenfunctions of $\A$ such that $\A e_k =\lambda_ke_k$,  for all $ k\in\mathbb{N}$.  Taking $\mfrak=\sum\limits_{j=1}^{\infty}(\mfrak,e_j)e_j$, we have  $\A\mfrak=\sum\limits_{j=1}^{\infty}\lambda_j(\mfrak,e_j)e_j$. Thus, it is immediate that 
	\begin{align*}
		\|\nabla\mfrak\|_{2}^2=\langle \A\mfrak,\mfrak\rangle =\sum_{j=1}^{\infty}\lambda_j|(\mfrak,e_j)|^2\geq \lambda_1\sum_{j=1}^{\infty}|(\mfrak,e_j)|^2=\lambda_1\|\mfrak\|_{2}^2,
	\end{align*}
	for all $\mfrak\in\V$, which is known as the \emph{Poincar\'e inequality}.   
\end{remark}

\begin{remark}
From \cite[Lemmas 1 and 2, Chapter I]{OAL}, we have the following well-known Ladyzhenskaya's inequality, for $d\in\{2,3\}$:
	\begin{align}\label{eqn-lady}
		\|\mfrak\|_{{4}} \leq 2^{\frac{d-1}{4}} \|\mfrak\|^{1-\frac{d}{4}}_{{2} } \|\nabla \mfrak\|^{\frac{d}{4}}_{{2}},\ \mfrak\in\H_0^1(\mathfrak{D}).
	\end{align}
	From the inequality \eqref{eqn-lady}, it is clear that  $\V\subset \wi \L^4$.
\end{remark}

\subsection{Trilinear form and bilinear operator}
We define the \emph{trilinear form} $b(\cdot,\cdot,\cdot):\V\times\V\times\V\to\R$ by $$b(\pfrak,\qfrak,\rfrak)=\int_{\mathfrak{D}}(\pfrak(x)\cdot\nabla)\qfrak(x)\cdot\rfrak(x)\d x=\sum_{i,j=1}^3\int_{\mathfrak{D}}\pfrak_i(x)\frac{\partial \qfrak_j(x)}{\partial x_i}\rfrak_j(x)\d x.$$ If $\pfrak, \qfrak$ are such that the linear map $b(\pfrak, \qfrak, \cdot) $ is continuous on $\V$, the corresponding element of $\V'$ is denoted by $\B(\pfrak, \qfrak)$. We represent   $\B(\pfrak) = \B(\pfrak, \pfrak)=\mathcal{P}(\pfrak\cdot\nabla)\pfrak$. An application of integration by parts gives
\begin{equation}\label{b0}
	\left\{
	\begin{aligned}
		b(\pfrak,\qfrak,\qfrak) &= 0,\ \text{ for all }\ \pfrak,\qfrak \in\V,\\
		b(\pfrak,\qfrak,\rfrak) &=  -b(\pfrak,\rfrak,\qfrak),\ \text{ for all }\ \pfrak,\qfrak,\rfrak \in \V.
	\end{aligned}
	\right.\end{equation}
It can be easily seen that $\B$ maps $\wi\L^4$ (and so $\V$) into $\V'$ and
\begin{align*}
	\left|\left<\B(\pfrak,\pfrak),\qfrak\right>\right|=\left|b(\pfrak,\qfrak,\pfrak)\right|\leq \|\pfrak\|_{4}^2\|\nabla\qfrak\|_{2},
\end{align*}
for all $\qfrak\in\V$, so that 
\begin{align}\label{2.9a}
	\|\B(\pfrak)\|_{\V'} \leq \|\pfrak\|_{4}^2 \leq C \|\pfrak\|_{\V}^2,\ \text{ for all }\ \pfrak\in\V.
\end{align}

\subsection{Nonlinear operator}
Let us now consider the operator $\mathcal{C}(\pfrak):=\mathcal{P}(|\pfrak|^2\pfrak)$.  From \cite[Remark 1.6]{Temam_1984}, the operator $\mathcal{C}(\cdot):\V \to\V^{\prime}$ is well-defined. Also, it is immediate that $\langle\mathcal{C}(\pfrak),\pfrak\rangle =\|\pfrak\|_{4}^{4}$. 
In addition, for  $\pfrak,\qfrak\in\V$,  we have (cf. \cite{Gautam+Mohan_2025})
\begin{align}\label{2.23}
	&\langle\mathcal{P}(\pfrak|\pfrak|^{2})-\mathcal{P}(\qfrak|\qfrak|^{2}),\pfrak-\qfrak\rangle\geq \frac{1}{2}\||\pfrak| (\pfrak-\qfrak)\|_{2}^2 + \frac{1}{2}\||\qfrak|(\pfrak-\qfrak)\|_{2}^2 \geq \frac14 	\|\pfrak-\qfrak\|_{4}^{4} \geq 0.
\end{align}


\section{State equation}\label{sec3}\setcounter{equation}{0}

In this section, we formulate the state equation in an abstract setting, introduce the notion of a weak solution, and provide well-posedness results together with Lipschitz continuity estimates.

\subsection{Abstract formulation and weak  solution}  Let us take  the projection $\mathcal{P}$ onto the system \eqref{1} to write down the abstract formulation of the system \eqref{1} as: 
\begin{equation}\label{eqn-CBF-projected}
	\left\{
	\begin{aligned}
		\frac{\d\mfrak(t)}{\d t}+\mu\A\mfrak(t)+\B(\mfrak(t))+\alpha\mfrak(t)+\beta\mathcal{C}(\mfrak(t))&=\f(t), \ t\in (0,T),\\
		\mfrak(0)&=\mfrak_0.
	\end{aligned}
	\right.
\end{equation}
For convenience, we have used $\f$ instead of $\mathcal{P}\f$. 
Let us now provide the definition of \emph{weak solutions} of the system \eqref{eqn-CBF-projected}. 
\begin{definition}\label{def-weak-sol}
	A function  $\mfrak\in\mathrm{C}([0,T];\H)\cap\mathrm{L}^2(0,T;\V)\cap\mathrm{L}^{4}(0,T;\wi\L^{4})$ with $\partial_t\mfrak\in\mathrm{L}^{2}(0,T;\mathbb{V}')+\mathrm{L}^{\frac{4}{3}}(0,T;\wi\L^{\frac{4}{3}}),$  is called a \emph{weak solution} to the system \eqref{eqn-CBF-projected}, if for $\f\in\mathrm{L}^2(0,T;\H)$, $\mfrak_0\in\H$ and $\v\in\V$, $\mfrak(\cdot)$ satisfies:
	\begin{equation}\label{3.13}
		\left\{
		\begin{aligned}
			\langle\partial_t\mfrak(t)+\mu \A \mfrak(t)+{\B}(\mfrak(t))+\alpha\mfrak(t)+\beta\mathcal{C}(\mfrak(t)),\v\rangle&=(\f(t) , \v),\\
			 \mfrak(0) & = \mfrak_0,
		\end{aligned}
		\right.
	\end{equation}
	for a.e. $t\in[0,T]$,
	and the following equality holds: 
	\begin{align}\label{eqn-energy-equality}
		&\|\mfrak(t)\|_{2}^2+2\mu\int_0^t\|\mfrak(s)\|_{\V}^2\d s+2\alpha\int_0^t\|\mfrak(s)\|_{2}^2\d s+2\beta\int_0^t\|\mfrak(s)\|_{{4}}^{4}\d s   = \|\mfrak_0\|_{2}^2+2\int_0^t (\f(s),\mfrak(s)) \d s,
	\end{align}
	for all $t\in[0,T]$.
\end{definition}

\subsection{Well-posedness and Lipschitz estimates}
It is known from \cite{Gautam+Mohan_2025} that the system \eqref{eqn-CBF-projected} admits a unique weak solution in the sense of Definition \ref{def-weak-sol}. Therefore, we state the following result which has been proved in \cite{Gautam+Mohan_2025}.

\begin{theorem} [{\cite[Theorem 3.4]{Gautam+Mohan_2025}}] \label{weake}
	Suppose that $2\beta\mu\geq 1$. 
	For $\mfrak_0\in\H$ and $
	\f\in\mathrm{L}^2(0,T;\H)$, there exists a \emph{unique weak solution} $\mfrak(\cdot)$  to the system \eqref{eqn-CBF-projected} in the sense of Definition \ref{def-weak-sol} and the following energy estimates holds:
	\begin{align}\label{eqn-sol-estimate}
		& \sup_{0\leq s\leq t} \|\mfrak(s)\|_{2}^2+2\mu\int_0^t\|\mfrak(s)\|_{\V}^2\d s+2\alpha\int_0^t\|\mfrak(s)\|_{2}^2\d s+2\beta\int_0^t\|\mfrak(s)\|_{{4}}^{4}\d s 
        \nonumber \\ &  \leq \left[\|\mfrak_0\|_{2}^2 + \int_0^t \|\f(s)\|_2^2 \d s\right]e^t
		 =: K_{t,\f,\mfrak_0},
	\end{align}
for all $t\in[0,T]$.
\end{theorem}

Next we obtain  Lipschitz estimates for the system \eqref{eqn-CBF-projected}.

\begin{proposition}\label{prop-stability-diff}
	Suppose that Hypothesis \ref{Para-Hypo} is satisfied. Let   $\mfrak_1$ and $\mfrak_2$ be two solutions to the system \eqref{3.13} sharing the initial data $\mfrak_0$,  and corresponding to external forcing $\f_1$ and $\f_2$, respectively. Then, the difference $\vfrak := \mfrak_1 - \mfrak_2$ satisfies the following   estimates:
	\begin{align}\label{eqn-diff-sol-estimate}
		&\sup_{0\leq t \leq T} \|\vfrak(t)\|_2^2 + 2 \mu(1-\kappa) \int_{0}^{T} \|\vfrak(t)\|^2_{\V} \d t + 2 \alpha \int_{0}^{T} \|\vfrak(t)\|_2^2 \d t     +  \frac{1}{2}\left(\beta - \frac{1}{2\mu\kappa}\right) \int_{0}^{T} \|\vfrak(t) \|_{4}^4 \d t
		\nonumber\\&
		\leq \sup_{0\leq t \leq T} \|\vfrak(t)\|_2^2 + 2 \mu(1-\kappa) \int_{0}^{T} \|\vfrak(t)\|^2_{\V} \d t + 2 \alpha \int_{0}^{T} \|\vfrak(t)\|_2^2 \d t     \nonumber\\ & \quad +  \left(\beta - \frac{1}{2\mu\kappa}\right) \int_{0}^{T} \left[\||\mfrak_1(t)| \vfrak(t) \|_{2}^2 + \||\mfrak_2(t)| \vfrak(t) \|_{2}^2\right] \d t
	\nonumber\\&
	\leq  e^T \int_{0}^{T}  \|\f_1(t) - \f_2(t)\|^2_2 \d t,
	\end{align}
where $\kappa$ is the constant appearing in Hypothesis \ref{Para-Hypo}.
\end{proposition}

\begin{proof}
	Note that the difference $\vfrak := \mfrak_1 - \mfrak_2$ satisfies the following system:
	\begin{equation}\label{eqn-difference}
		\left\{
		\begin{aligned}
			\frac{\d\vfrak}{\d t} +\mu\A\vfrak+\B(\mfrak_1, \vfrak) + \B(\vfrak, \mfrak_2) + \alpha\vfrak   + \frac{\beta}{2} \mathcal{P} \{ [|\mfrak_1|^2 + |\mfrak_2|^2] \vfrak \}  & \\ + \frac{\beta}{2}\mathcal{P}\{ [(\mfrak_1+\mfrak_2)\cdot\vfrak](\mfrak_1+\mfrak_2)\}  
			 &=\f_1 - \f_2, \\
			\vfrak(0)&=\boldsymbol{0}.
		\end{aligned}
		\right.
	\end{equation}
	Taking the inner product of equation $\eqref{eqn-difference}_1$ with $\vfrak$, we obtain
	\begin{align}\label{37}
	& 	\frac{1}{2} \frac{\d }{\d t}\|\vfrak(t)\|_2^2 + \mu\|\vfrak(t)\|^2_{\V} + \alpha \|\vfrak(t)\|_2^2   + \frac{\beta}{2}\left[ \||\mfrak_1(t)|\vfrak(t)\|^2_2 + \||\mfrak_2(t)|\vfrak(t)\|^2_2\right] 
	\nonumber \\ & + \frac{\beta}{2}  \| (\mfrak_1(t)  + \mfrak_2(t))\cdot \vfrak(t)\|_2^2 
	\nonumber\\ & = - b(\vfrak(t), \mfrak_2(t), \vfrak(t))   + (\f_1(t)-\f_2(t), \vfrak(t)),
	\end{align}
for a.e. $t\in[0,T]$.  Using H\"older's and Young's inequalities, we estimate $|b(\w,\u,\w)|$ as  
\begin{align}
	|b(\vfrak, \mfrak_2 , \vfrak)| & \leq 
		\||\mfrak_2|\vfrak\|_{2}\|\vfrak\|_{\V}  \leq \mu\kappa \|\vfrak\|_{\V}^2 + \frac{1}{4\mu\kappa} \| |\mfrak_2 |\vfrak\|^2_{2}, \label{38} \\
|(\f_1 - \f_2, \vfrak)| & \leq \|\f_1 - \f_2\|_2 \|\vfrak\|_2  \leq \frac12 \|\f_1 - \f_2\|^2_2 +  \frac12 \|\vfrak\|^2_2,\label{39}
\end{align}
where $\kappa$ is the same as in Hypothesis \ref{Para-Hypo}.
Combining \eqref{37}-\eqref{39}, we find
\begin{align}
	& \frac{\d }{\d t} \|\vfrak(t)\|_2^2 + 2 \mu(1-\kappa)\|\vfrak(t)\|^2_{\V} + 2 \alpha \|\vfrak(t)\|_2^2     +  \left(\beta - \frac{1}{2\mu\kappa}\right)\left[\||\mfrak_1| \vfrak \|_{2}^2 + \||\mfrak_2| \vfrak \|_{2}^2\right]
	\nonumber\\ &  \leq   \|\vfrak(t)\|^2_2 +   \|\f_1 (t)- \f_2(t)\|^2_2,
\end{align}
for a.e. $t\in[0,T]$.  An application of Gronwall's inequality completes the proof.
\end{proof}


\section{Intermediate adjoint system}\label{sec4}\setcounter{equation}{0}
To derive the first-order necessary optimality conditions, we are required to investigate the following intermediate adjoint system:
\begin{equation}\label{eqn-adjoint}
	\left\{
	\begin{aligned}
		-\frac{\d\qfrak}{\d t} + \mu \A \qfrak - \B(\mfrak_1, \qfrak) + \mathcal{P}[ \sum_{j=1}^3 [\nabla(\mfrak_2)_j]\qfrak_j ]  + \alpha\qfrak  & + \frac{\beta}{2} \mathcal{P}\{ [|\mfrak_1|^2 + |\mfrak_2|^2] \qfrak \}  \\ + \frac{\beta}{2}\mathcal{P}\{ [(\mfrak_1+\mfrak_2)\cdot\qfrak](\mfrak_1+\mfrak_2)\}  & = \mathcal{P}\h, \\
		\qfrak(T)&=\boldsymbol{0},
	\end{aligned}
	\right.
\end{equation}
where   $\mfrak_1$ and $\mfrak_2$ are two solutions to the system \eqref{3.13} sharing the initial data $\mfrak_0$,  and corresponding to external forcing $\f_1$ and $\f_2$, respectively.  Let us give the definition of weak solution of system \eqref{eqn-adjoint}.

\begin{definition}\label{def-adjoint}
	A function  $\qfrak\in\mathrm{L}^{\infty}(0,T;\H)\cap\mathrm{L}^2(0,T;\V)$ with  $\partial_t\qfrak\in\mathrm{L}^{2}(0,T;\mathbb{V}')+\mathrm{L}^{\frac{4}{3}}(0,T;\wi\L^{\frac{4}{3}})$, 
	\begin{align*}
	\int_{0}^{T}	\int_{\mathfrak{D}}  [|\mfrak_1(x,t)|^{2}+|\mfrak_2(x,t)|^{2}]|\qfrak(x, t)|^2   \d x \d t 
< + \infty,
	\end{align*}
and $\qfrak(T)=\boldsymbol{0}$,	 is called a \emph{weak solution} to the system \eqref{eqn-adjoint}, if for $\h \in\mathrm{L}^2(0,T;\L^2(\mathfrak{D}))$ and $\v\in\V$, $\qfrak(\cdot)$ satisfies:
	\begin{equation}\label{eqn-adjoint-WF}	 
		\begin{aligned}
			\left\langle-\frac{\d\qfrak}{\d t} + \mu \A \qfrak - \B(\mfrak_1, \qfrak) + \mathcal{P}\left\{ \sum_{j=1}^3 [\nabla(\mfrak_2)_j]\qfrak_j \right\}  + \alpha\qfrak \right.  & \left. + \frac{\beta}{2} \mathcal{P}\left\{ [|\mfrak_1|^2 + |\mfrak_2|^2] \qfrak \right\} \right. \\  \left. + \frac{\beta}{2}\mathcal{P}\left\{ [(\mfrak_1+\mfrak_2)\cdot\qfrak](\mfrak_1+\mfrak_2)\right\} ,\v\right\rangle&=(\h,\v),
		\end{aligned}
 	\end{equation}
for a.e. $t\in[0,T]$. 
	
\end{definition}

\subsection{Regularization of the system \eqref{eqn-adjoint}}
In order to find the existence of a weak solution of system \eqref{eqn-adjoint} in the sense of Definition \ref{def-adjoint}, we  first prove the existence of a unique weak solution of the following regularized version of system \eqref{eqn-adjoint}:
\begin{equation}\label{eqn-adjoint-regular}
	\left\{
	\begin{aligned}
		-\frac{\d\qfrak^{\delta}}{\d t} + \mu \A \qfrak^{\delta} + \delta \mathcal{C}(\qfrak^{\delta}) - \B(\mfrak_1, \qfrak^{\delta}) + \mathcal{P}\left\{ \sum_{j=1}^3 [\nabla(\mfrak_2)_j](\qfrak^{\delta})_j \right\}  + \alpha\qfrak^{\delta}    &   + \frac{\beta}{2} \mathcal{P}\left\{ [|\mfrak_1|^2 + |\mfrak_2|^2] \qfrak^{\delta} \right\}  \\    + \frac{\beta}{2}\mathcal{P}\{ [(\mfrak_1+\mfrak_2)\cdot\qfrak^{\delta}](\mfrak_1+\mfrak_2)\}  & = \mathcal{P} \h, \\
		\qfrak^{\delta}(T)&=\boldsymbol{0}.
	\end{aligned}
	\right.
\end{equation} 
Next we provide the definition of weak solution to   system \eqref{eqn-adjoint-regular}.
\begin{definition}\label{def-adjoint-reg}
	A function  $\qfrak^{\delta}\in\mathrm{L}^{\infty}(0,T;\H)\cap\mathrm{L}^2(0,T;\V)$ with  $\partial_t\qfrak^{\delta}\in\mathrm{L}^{2}(0,T;\mathbb{V}')+\mathrm{L}^{\frac{4}{3}}(0,T;\wi\L^{\frac{4}{3}})$, 
	\begin{align*}
		\int_{0}^{T}	\int_{\mathfrak{D}}  [|\mfrak_1(x,t)|^{2}+|\mfrak_2(x,t)|^{2}]|\qfrak^{\delta}(x, t)|^2   \d x \d t 
		< + \infty,
	\end{align*}
	and $\qfrak^{\delta}(T)=\boldsymbol{0}$,	 is called a \emph{weak solution} to the system \eqref{eqn-adjoint}, if for $\h \in\mathrm{L}^2(0,T;\L^2(\mathfrak{D}))$ and $\v\in\V$, $\qfrak^{\delta}(\cdot)$ satisfies:
	\begin{equation}\label{eqn-adjoint-reg-WF}	 
		\begin{aligned}
			\left\langle-\frac{\d\qfrak^{\delta}}{\d t} + \mu \A \qfrak^{\delta}  + \delta \mathcal{C}(\qfrak^{\delta}) - \B(\mfrak_1, \qfrak^{\delta}) + \mathcal{P}\left\{ \sum_{j=1}^3 [\nabla(\mfrak_2)_j](\qfrak^{\delta})_j \right\}  + \alpha\qfrak^{\delta} \right.  & \left. + \frac{\beta}{2} \mathcal{P}\left\{ [|\mfrak_1|^2 + |\mfrak_2|^2] \qfrak^{\delta} \right\} \right. \\  \left. + \frac{\beta}{2}\mathcal{P}\left\{ [(\mfrak_1+\mfrak_2)\cdot\qfrak^{\delta}](\mfrak_1+\mfrak_2)\right\} ,\v\right\rangle&=(\h,\v),
		\end{aligned}
	\end{equation}
	for a.e. $t\in[0,T]$. 
\end{definition}

\begin{theorem}\label{thm-sol-delta}
For $\h \in\mathrm{L}^2(0,T;\L^2(\mathfrak{D}))$,	there exists a unique weak solution $\qfrak^\delta$ to system \eqref{eqn-adjoint-regular} in the sense of Definition \ref{def-adjoint-reg}. In addition, we have
\begin{align}\label{eqn-opt-con-delta}
	&   \int_{0}^{T} (\f_1(t) - \f_2(t) , \qfrak^\delta(t)) \d t + \delta  \int_{0}^{T}   \langle \mathcal{C}(\qfrak^\delta(t)) , \vfrak(t) \rangle \d t
	  =   \int_{0}^{T}  (\h(t) , \mfrak_1 (t) - \mfrak_2 (t)) \d t  .
\end{align}
and the following estimates hold:
\begin{align}\label{eqn-energy-estimates-delta}
	&   \sup_{t\in[0,T]}\|\qfrak^\delta(t)\|_2^2 + 2\mu(1-\kappa)\int_0^T \|\qfrak^\delta(t)\|^{2}_{\V} \d t + 2 \delta  \int_0^T \|\qfrak^\delta(t)\|^{4}_{4}\d t 
	\nonumber\\ &	+  \left(\beta-\frac{1}{2\mu\kappa}\right) \int_0^T [ \||\mfrak_1(t)| \qfrak^\delta(t)\|_2^2 +  \|| \mfrak_1(t)| \qfrak^\delta(t)\|_2^2]   \d t
	\leq   e^T \int_0^T \| \h(t)\|^2_2 \d t, 
\end{align} 
and there exists a constant $\widehat{K}_{\h, \mu,\alpha,\beta,\kappa,T}$ independent of $\delta$ such that 
\begin{align}\label{eqn-derivative-estimates-delta}
	\left\|\frac{\d \qfrak^\delta}{\d t} \right\|_{\Lrm^{2}(0,T;\V^{\prime})+\Lrm^{\frac43}(0,T;\wi\L^\frac43)} \leq \left[\widehat{K}_{\h, \mu,\alpha,\beta,\kappa,T}+ \delta^{\frac14} \left(\frac{\wi K_{\h,T}}{2}\right)^{\frac34}\right],
\end{align}
where $\wi K_{\h,T} = e^T \int_0^T \| \h(t)\|^2_2 \d t$.
\end{theorem}

\begin{proof}
Let us first note that $\qfrak^{\delta}$ is the solution of the terminal value problem \eqref{eqn-adjoint-regular}, if and only if $\pfrak^{\delta}(t) = \qfrak^{\delta}(T-t)$ is the solution of the following initial value problem:
\begin{equation}\label{eqn-adjoint-regular-forward}
	\left\{
	\begin{aligned}
		\frac{\d\pfrak}{\d t} + \mu \A \pfrak + \delta \mathcal{C}(\pfrak) - \B(\wi\mfrak_1, \pfrak) + \mathcal{P}[ \sum_{j=1}^3 [\nabla(\wi\mfrak_2)_j]\pfrak_j ]  + \alpha\pfrak  &   + \frac{\beta}{2} \mathcal{P}\{ [|\wi\mfrak_1|^2 + |\wi\mfrak_2|^2] \pfrak \}  \\    + \frac{\beta}{2}\mathcal{P}\{ [( \wi\mfrak_1 + \wi\mfrak_2)\cdot\pfrak](\wi\mfrak_1 + \wi \mfrak_2)\}  & = \wi \h, \\
		\pfrak(0) & = \boldsymbol{0},
	\end{aligned}
	\right.
\end{equation} 
where $\wi\mfrak_1(t)= \mfrak_1(T-t)$, $\wi\mfrak_2(t)= \mfrak_2(T-t)$ and $\wi\h(t)= \mathcal{P} \h(T-t)$.

Consider the finite-dimensional space $\H_{n} = \mathrm{span}\lbrace e_{1},\dots,e_{n}\rbrace$ and define the orthogonal projection $\P_n$ from $\H$ to $\H_n$ as
\begin{equation*}
	\P_n : \H \to \H_n; \quad  \H\ni \u \mapsto \P_n \u = \sum_{j=1}^{n} (\u,e_{j})e_{j}.
\end{equation*}
Notice that $\P_n$ is self-adjoint, $\|\P_n \u\|_2 \leq \|\u\|_2$ for any $\u\in \H$. Thus, the dominated convergence theorem yields, $\P_n \u \to \u,$ as $n \to \infty,$ in $\H$.

Also, we derive an orthonormal basis in $\V$, given as $\left\lbrace\hat{e}_j={\frac{1}{\sqrt{\lambda_j}}} e_j\right\rbrace_{j \in \mathbb{N}}$. Moreover, $\P_n \v = \sum\limits_{j=1}^{n} (\v,e_{j})e_{j} = \sum\limits_{j=1}^{n} (\v,\hat{e}_{j})_{\V}\hat{e}_{j}$, where $(\cdot,\cdot)_{\V}$ represents the inner product of Hilbert space $\V$, and thus $\|\P_n \v\|_{\V} \leq \|\v\|_{\V}$.


For each $n = 1,2,\ldots$, we search for an approximate solution of
the form
\begin{align*}
	\pfrak_n (x,t) = \sum_{j=1}^n w_j^n(t) e_j(x),
\end{align*}
where $w_1^n(t), \ldots, w_j^n(t)$ are unknown scalar functions of $t$ such that it satisfies the following finite-dimensional system of ordinary differential equations in $\H_n$:
\begin{equation}\label{eqn-finite-D}
	\left\{
\begin{aligned}
	\left( \frac{\d\pfrak_n(t)}{\d t} + \mu \P_n \A \pfrak_n(t) + \delta\P_n\mathcal{C}(\pfrak_n(t)) \right.   - \P_n \B(\wi\mfrak_1(t), \pfrak_n(t)) & + \P_n \mathcal{P}[ \sum_{j=1}^3 [\nabla(\wi\mfrak_2(t))_j](\pfrak_n(t))_j ] \\  + \alpha\pfrak_n(t)  
	   + \frac{\beta}{2} \P_n \mathcal{P}\{ [|\wi\mfrak_1(t)|^2 & + |\wi\mfrak_2(t)|^2] \pfrak_n(t) \}     
       \\ \left.  +  \frac{\beta}{2} \P_n \mathcal{P}\{ [( \wi\mfrak_1(t) + \wi\mfrak_2(t))\cdot\pfrak_n(t) ](\wi\mfrak_1(t) + \wi \mfrak_2(t))\}, e_j \right)   & = ( \P_n \wi \h(t), e_j), \\
	(\pfrak_n(0), e_j) & = 0, 
\end{aligned}
\right.
\end{equation}
for a.e. $t\in [0,T]$ and $j=1,\ldots,n$.   Since the coefficients of system \eqref{eqn-finite-D} are locally Lipschitz, therefore by using the Carath\'eodory’s existence theorem, there exists a local maximal solution $\pfrak_n \in \mathrm{C}([0, T^{\ast}_n]; \H_n)$, for some $0<T^{\ast}_n \leq T$ of the system \eqref{eqn-finite-D} and uniqueness is immediate from the local Lipschitz property. The time $T^{\ast}_n$ can be extended to $T$ by establishing the uniform energy estimates of the solutions satisfied by the system \eqref{eqn-finite-D}.

\vskip 2mm
\noindent
\textit{A priori estimates for $\pfrak_n$:} Multiplying $\eqref{eqn-finite-D}_1$ by $w_j^n(\cdot)$, summing up over $j=1,\ldots, n$, we get
\begin{align}\label{322}
	& \frac12 \frac{\d}{\d t}\|\pfrak_n(t)\|_2^2 + \mu \|\pfrak_n(t)\|^{2}_{\V} + \delta \|\pfrak_n(t)\|^{4}_{4} + \alpha \|\pfrak_n(t)\|^{2}_{2}  
	   +  \frac{\beta}{2} \int_{\mathfrak{D}}  [|\wi\mfrak_1(x,t)|^{2}+|\wi\mfrak_2(x,t)|^{2}]|\pfrak_n(x,t)|^2    \d x
	\nonumber\\ &   + \frac{\beta}{2} \int_{\mathfrak{D}} [\{\wi\mfrak_1(x,t) + \wi\mfrak_2(x,t)\} \cdot\pfrak_n(x,t)]^2  \d x
	  = - b(\pfrak_n(t), \wi\mfrak_2(t), \pfrak_n(t)) + (\wi\h (t), \pfrak_n(t)),
\end{align} 
for a.e. $t\in[0,T]$. Performing H\"older's and Young's inequalities, we have
\begin{align}
	\left| b(\pfrak_n, \wi\mfrak_2, \pfrak_n) \right| & =  \left| b(\pfrak_n, \pfrak_n, \wi\mfrak_2) \right| \leq \|\pfrak_n\|_{\V} \||\wi\mfrak_2|\pfrak_n\|_2 \leq  \mu\kappa \|\pfrak_n\|^2_{\V} +  \frac{1}{4\mu\kappa} \||\wi\mfrak_2|\pfrak_n\|^2_2,\label{323} \\
	|(\wi\h , \pfrak_n)|  & \leq \|\wi\h\|_2 \|\pfrak_n\|_2 \leq \frac12 \|\wi\h\|^2_2  +  \frac12  \|\pfrak_n\|^2_2,\label{324}
 \end{align}
where $\kappa$ is the same as in Hypothesis \eqref{Para-Hypo}. Combining \eqref{322}-\eqref{324}, we reach at
\begin{align}\label{3.26}
	&  \frac{\d}{\d t}\|\pfrak_n(t)\|_2^2 + 2\mu(1-\kappa) \|\pfrak_n(t)\|^{2}_{\V}  +   2\alpha \|\pfrak_n(t)\|^{2}_{2} + 2 \delta \|\pfrak_n(t)\|^{4}_{4} 
 	\nonumber\\ & +  \left(\beta-\frac{1}{2\mu\kappa}\right) [ \||\wi\mfrak_1(t)| \pfrak_n(t)\|_2^2 +  \||\wi\mfrak_1(t)| \pfrak_n(t)\|_2^2]  
	\nonumber\\ &  \leq  \|\wi\h(t)\|^2_2  +    \|\pfrak_n(t)\|^2_2,
\end{align} 
for a.e. $t\in[0,T]$. This implies, in view of Gronwall's inequality, that
\begin{align}\label{eqn-energy-estimates-reqular}
	&   \sup_{t\in[0,T]}\|\pfrak_n(t)\|_2^2 + 2\mu(1-\kappa)\int_0^T \|\pfrak_n(t)\|^{2}_{\V} \d t  + 2\alpha \int_0^T \|\pfrak_n(t)\|^{2}_{2} \d t  + 2 \delta  \int_0^T \|\pfrak_n(t)\|^{4}_{4}\d t 
  \nonumber\\ &	+  \left(\beta-\frac{1}{2\mu\kappa}\right) \int_0^T [ \||\wi\mfrak_1(t)| \pfrak_n(t)\|_2^2 +  \||\wi\mfrak_1(t)| \pfrak_n(t)\|_2^2]   \d t
\nonumber\\ & 	\leq   e^T \int_0^T \|\wi\h(t)\|^2_2 \d t \leq     e^T \int_0^T \| \h(t)\|^2_2 \d t:=  \wi K_{\h,T} .
\end{align} 
Therefore, we find 
\begin{align}\label{eqn-weak-con-1}
	\{\pfrak_n\}_{n\in\N} \text{ is uniformly bounded in } \Lrm^{\infty}(0,T;\H)\cap \Lrm^{2}(0,T;\V) \cap \Lrm^{4}(0,T;\wi\L^4),
\end{align}
\begin{align}\label{eqn-weak-con-2}
	\{|\pfrak_n|^2\pfrak_n\}_{n\in\N} \text{ is uniformly bounded in }  \Lrm^{\frac43}(0,T;\L^{\frac43}(\mathfrak{D})),
\end{align}
and
\begin{align}\label{eqn-weak-con-3}
	\{|\mfrak_1|\pfrak_n\}_{n\in\N}, \{|\mfrak_2|\pfrak_n\}_{n\in\N} \text{ are uniformly bounded in }  \Lrm^{2}(0,T;\L^{2}(\mathfrak{D})).
\end{align}

\vskip 2mm
\noindent
\textit{A priori estimates for $\partial_t\pfrak_n$}: For any $\boldphi \in \Lrm^4(0,T; \V) $, we have
\begin{align*}
& \left|	\left\langle \frac{\d \pfrak_n}{\d t} , \boldphi \right\rangle \right|
\nonumber\\	&=  \bigg| \bigg\langle -\mu \P_n \A \pfrak_n - \delta \P_n\mathcal{C}(\pfrak_n)    + \P_n \B(\wi\mfrak_1, \pfrak_n)  - \P_n \mathcal{P}[ \sum_{j=1}^3 [\nabla(\wi\mfrak_2)_j](\pfrak_n)_j ]  - \alpha\pfrak_n  
	\nonumber  \\  & \qquad -  \frac{\beta}{2} \P_n \mathcal{P}\{ [|\wi\mfrak_1|^2 + |\wi\mfrak_2|^2] \pfrak_n \}      -  \frac{\beta}{2} \P_n \mathcal{P}\{ [( \wi\mfrak_1 + \wi\mfrak_2)\cdot\pfrak_n ](\wi\mfrak_1 + \wi \mfrak_2)\} + \P_n \wi\h , \boldphi \bigg\rangle \bigg|
	\nonumber\\ &=  \bigg| - \mu (\nabla \pfrak_n, \nabla \P_n\boldphi) -  \delta \langle  \mathcal{C}(\pfrak_n), \P_n\boldphi \rangle   + b(\wi\mfrak_1, \pfrak_n, \P_n\boldphi)   - b(\pfrak_n, \wi\mfrak_2,  \P_n\boldphi)  - (\alpha\pfrak_n - \wi\h , \P_n\boldphi)
	\nonumber  \\  & \qquad - \left\langle  \frac{\beta}{2}  \{ [|\wi\mfrak_1|^2 + |\wi\mfrak_2|^2] \pfrak_n \}      +  \frac{\beta}{2}  \{ [( \wi\mfrak_1 + \wi\mfrak_2)\cdot\pfrak_n ](\wi\mfrak_1 + \wi \mfrak_2)\}  , \P_n \boldphi \right\rangle \bigg|
	\nonumber\\ & \leq C\bigg[\|\pfrak_n\|_{\V} + \||\wi\mfrak_1|\pfrak_n\|_2 + \||\wi\mfrak_2|\pfrak_n\|_2 + \|\wi\h\|_2 \bigg]\|\P_n\boldphi\|_{\V} + \delta\|\pfrak_n\|^3_{4} \|\P_n\boldphi\|_4 
	\nonumber\\ & \quad + C \bigg[\||\wi\mfrak_1|\pfrak_n\|_2 \|\wi\mfrak_1\|_4 + \||\wi\mfrak_2|\pfrak_n\|_2 \|\wi\mfrak_2\|_4 \bigg]\|\P_n\boldphi\|_4
	\nonumber\\ & \leq C\bigg[\|\pfrak_n\|_{\V} + \||\wi\mfrak_1|\pfrak_n\|_2 + \||\wi\mfrak_2|\pfrak_n\|_2 + \|\wi\h\|_2 + \delta\|\pfrak_n\|^3_{4}  
	 + \||\wi\mfrak_1|\pfrak_n\|_2 \|\wi\mfrak_1\|_4 
	   + \||\wi\mfrak_2|\pfrak_n\|_2 \|\wi\mfrak_2\|_4 \bigg]\|\P_n\boldphi\|_{\V}
	   \nonumber\\ & \leq C\bigg[\|\pfrak_n\|_{\V} + \||\wi\mfrak_1|\pfrak_n\|_2 + \||\wi\mfrak_2|\pfrak_n\|_2 + \|\wi\h\|_2 + \delta\|\pfrak_n\|^3_{4}  
	   + \||\wi\mfrak_1|\pfrak_n\|_2 \|\wi\mfrak_1\|_4 
	   + \||\wi\mfrak_2|\pfrak_n\|_2 \|\wi\mfrak_2\|_4 \bigg]\|\boldphi\|_{\V}, 
\end{align*}
which implies
\begin{align}\label{eqn-derivative}
	& \left| \int_{0}^{T} \left\langle \frac{\d \pfrak_n(t)}{\d t} , \boldphi (t) \right\rangle \d t \right| 
	\nonumber\\ & \leq C \int_{0}^{T} \bigg[\|\pfrak_n(t)\|_{\V} + \||\wi\mfrak_1(t)|\pfrak_n(t)\|_2 + \||\wi\mfrak_2(t)|\pfrak_n(t)\|_2 + \|\wi\h(t)\|_2 + \delta\|\pfrak_n(t)\|^3_{4}  
	\nonumber\\ & \quad  + \||\wi\mfrak_1(t)|\pfrak_n(t)\|_2 \|\wi\mfrak_1(t)\|_4 
	+ \||\wi\mfrak_2(t)|\pfrak_n(t)\|_2 \|\wi\mfrak_2(t)\|_4 \bigg]\|\boldphi(t)\|_{\V} \d t
\nonumber\\ & \leq C \bigg[ T^{\frac{1}{4}}\|\pfrak_n\|_{\Lrm^2(0,T;\V)} + T^{\frac{1}{4}}\|\wi\h\|_{\Lrm^2(0,T;\H)} + \delta \|\pfrak_n\|^3_{\Lrm^4(0,T;\wi \L^4)} + T^{\frac14}\left(\int_{0}^{T}\||\wi\mfrak_1(t)|\pfrak_n(t)\|_2^2\d t\right)^{\frac12} 
\nonumber\\ & \quad + T^{\frac14} \left(\int_{0}^{T}\||\wi\mfrak_2(t)|\pfrak_n(t)\|_2^2\d t\right)^{\frac12} 
 + \left(\int_{0}^{T}\||\wi\mfrak_1(t)|\pfrak_n(t)\|_2^2\d t\right)^{\frac12}\|\wi\mfrak_1\|_{\Lrm^4(0,T;\wi \L^4)}
 \nonumber\\ & \quad + \left(\int_{0}^{T}\||\wi\mfrak_2(t)|\pfrak_n(t)\|_2^2\d t\right)^{\frac12}\|\wi\mfrak_2\|_{\Lrm^4(0,T;\wi \L^4)}  \bigg] \|\boldphi\|_{\Lrm^4(0,T; \V)}. 
\end{align}
Making use of \eqref{eqn-energy-estimates-reqular} (uniform bounds), $\wi\h \in \Lrm^{2}(0,T;\H)$ and the fact that $\wi\mfrak_1, \wi\mfrak_2 \in \Lrm^{4}(0,T;\wi\L^4)$ in \eqref{eqn-derivative}, we infer that 
\begin{align}
	\{\partial_t \pfrak_n\}_{n\in\N} \;\; \text{ is uniformly bounded in } \; \Lrm^{\frac43}(0,T; \V^{\prime}).
\end{align}

\vskip 2mm
\noindent
\textit{Passing to the limit:} Using \eqref{eqn-weak-con-1}-\eqref{eqn-weak-con-3} and the  {Banach Alaoglu theorem}, we infer the existence of an element $\pfrak\in\mathrm{L}^{\infty}(0,T;\H)\cap\mathrm{L}^{2}(0,T;\V)\cap\mathrm{L}^{4}(0,T;\widetilde{\L}^{4})$ with $\frac{\d \pfrak}{\d t}\in \mathrm{L}^{\frac43}(0,T;\V')$ and $\boldsymbol{\xi}_1 , \boldsymbol{\xi}_2, \boldsymbol{\xi}_3\in \Lrm^{2}(0,T;\L^{2}(\mathfrak{D}))$ such that
\begin{align}
	\pfrak_n\xrightharpoonup{w^*}&\ \pfrak  &&\text{ in }	\mathrm{L}^{\infty}(0,T;\H),\label{S7}\\
	\pfrak_n\xrightharpoonup{w}&\ \pfrak  && \text{ in } \mathrm{L}^{2}(0,T;\V)\cap\mathrm{L}^{4}(0,T;\widetilde{\L}^{4}),\label{S8}\\
	|\pfrak_n|^2\pfrak_n  \xrightharpoonup{w}&\ \boldsymbol{\xi}  && \text{ in }  \Lrm^{\frac43}(0,T;\L^{\frac43}(\mathfrak{D})) \label{S88}\\
	|\mfrak_1|\pfrak_n  \xrightharpoonup{w}&\ \boldsymbol{\xi}_1  && \text{ in }  \Lrm^{2}(0,T;\L^{2}(\mathfrak{D})) \label{S89}\\
	|\mfrak_2|\pfrak_n  \xrightharpoonup{w}&\ \boldsymbol{\xi}_2  && \text{ in }  \Lrm^{2}(0,T;\L^{2}(\mathfrak{D})) \label{S810}\\
	( \wi\mfrak_1 + \wi\mfrak_2)\cdot\pfrak_n  \xrightharpoonup{w}&\ \boldsymbol{\xi}_3  && \text{ in }  \Lrm^{2}(0,T;\L^{2}(\mathfrak{D})) \label{S811}\\
	\frac{\d \pfrak_n}{\d t}\xrightharpoonup{w}&\frac{\d \pfrak}{\d t}  && \text{ in }\mathrm{L}^{\frac{4}{3}}(0,T;\V'),\label{S8d}
\end{align}
along a subsequence. Since $\pfrak_n\in\mathrm{L}^{2}(0,T;\V)$ and $\frac{\d \pfrak_n}{\d t}\in \mathrm{L}^{\frac{4}{3}}(0,T;\V')$, the embeddings $\V \subset \H \subset\V^{\prime}$ are continuous  and the embedding $\V \subset \H$ is compact, then  the  {Aubin-Lions compactness lemma} implies that (see \cite[Chapter III, Theorem 2.1, p.184]{Temam_1984})
\begin{align}\label{S9}
	\pfrak_n\to\pfrak \ \text{ strongly in } \ \mathrm{L}^2(0,T;\H).
\end{align}
This also implies that 
\begin{align}\label{S10}
	\pfrak_n (x,t) \to\pfrak (x,t) \ \text{ for a.e. } \ (x,t)\in [0,T]\times\mathfrak{D},
\end{align}
along a further subsequence. It provides, in view of \cite[Lemma 1.3]{Lions_1969} and \eqref{S88}-\eqref{S810}, that 
\begin{align}
	|\pfrak_n|^2\pfrak_n  \xrightharpoonup{w}&\ |\pfrak|^2\pfrak \;  \;\;\; && \text{ in } \;\;\; \Lrm^{\frac43}(0,T; \L^{\frac43}(\mathfrak{D})), \label{S11}\\
	|\wi\mfrak_1|\pfrak_n  \xrightharpoonup{w}&\ |\wi\mfrak_1|\pfrak  \;\;\;  && \text{ in }   \;\;\; \Lrm^{2}(0,T;\L^{2}(\mathfrak{D})), \label{S12}\\
	|\wi\mfrak_2|\pfrak_n  \xrightharpoonup{w}&\ |\wi\mfrak_2|\pfrak   \;\;\; && \text{ in }  \;\;\;  \Lrm^{2}(0,T;\L^{2}(\mathfrak{D})), \label{S13}\\
	( \wi\mfrak_1 + \wi\mfrak_2)\cdot\pfrak_n  \xrightharpoonup{w}&\ ( \wi\mfrak_1 + \wi\mfrak_2)\cdot\pfrak \;\;\; && \text{ in } \;\;\; \Lrm^{2}(0,T;\L^{2}(\mathfrak{D})), \label{S14}
\end{align}
which also tell
\begin{align}
	\int_0^T \int_{\mathfrak{D}}  [|\wi\mfrak_1(x,t)|^{2} + |\wi\mfrak_2(x,t)|^{2}]|\pfrak(x,t)|^2    \d x \d t < + \infty.
\end{align}
Using the convergences \eqref{S8}, \eqref{S9} and \eqref{S11}-\eqref{S13}, and the fact that $\P_n\wi\h \to \wi\h$ in $\Lrm^2(0,T;\H)$, we obtain that $\pfrak$ satisfies:
\begin{align}\label{eqn-weak-sol-forward}
	&\int_{0}^{T}\left\langle \frac{\d\pfrak(t)}{\d t} + \mu \A \pfrak(t) + \delta \mathcal{C}(\pfrak(t)) \right.   -  \B(\wi\mfrak_1(t), \pfrak(t))  +  \mathcal{P}\left\{ \sum_{j=1}^3 [\nabla(\wi\mfrak_2(t))_j](\pfrak(t))_j \right\}  + \alpha\pfrak(t)  
	\nonumber \\ & \left. + \frac{\beta}{2}  \mathcal{P}\left\{ [|\wi\mfrak_1(t)|^2 + |\wi\mfrak_2(t)|^2] \pfrak(t) \right\}      +  \frac{\beta}{2}  \mathcal{P}\{ [( \wi\mfrak_1(t) + \wi\mfrak_2(t))\cdot\pfrak(t) ](\wi\mfrak_1(t) + \wi \mfrak_2(t))\}, \v  \right\rangle g(t)  \d t   
    \nonumber\\ & = \int_{0}^{T}( \wi \h(t), \v) g(t) \d t, 
\end{align}
for all $g\in C_c^{\infty}([0,T))$ and $\v\in \V$. A calculation similar to \eqref{eqn-derivative} provides that $\frac{\d \pfrak}{\d t}\in \mathrm{L}^{2}(0,T;\V')+ \mathrm{L}^{\frac43}(0,T;\wi\L^{\frac43})$. 

By  \cite[Chapter II, Theorem 1.8]{CV}, $\pfrak\in\mathrm{L}^{2}(0,T;\V)\cap\mathrm{L}^{4}(0,T;\widetilde{\L}^{4})$ and $\frac{\d \pfrak}{\d t}\in\mathrm{L}^{2}(0,T;\V')+\mathrm{L}^{\frac{4}{3}}(0,T;\widetilde{\L}^{\frac{4}{3}})$ imply that $\pfrak\in\mathrm{C}([0,T];\H)$ and satisfies the following the energy equality:
\begin{align}\label{eqn-eeq}
	& \frac12 \frac{\d}{\d t}\|\pfrak(t)\|_2^2 + \mu \|\pfrak(t)\|^{2}_{\V} + \delta \|\pfrak(t)\|^{4}_{4} + \alpha \|\pfrak(t)\|^{2}_{2}  
	\nonumber\\ &  +  \frac{\beta}{2}  [ \||\wi\mfrak_1(t)| \pfrak(t)\|_2^2 +  \||\wi\mfrak_1(t)| \pfrak(t)\|_2^2 + \|\{\wi\mfrak_1(t) + \wi\mfrak_2(t)\} \cdot\pfrak(t)\|_2^2]   
	\nonumber\\ & = - b(\pfrak(t), \wi\mfrak_2(t), \pfrak(t)) + (\wi\h (t), \pfrak(t)),
\end{align} 
for a.e. $t\in[0,T]$. A standard argument also gives that $\pfrak(0)=\boldsymbol{0}$ (see e.g. \cite{GPG}). Hence, we conclude that there exists a weak solution of system \eqref{eqn-adjoint-regular-forward}. 

\vskip 2mm
\noindent
\textit{Uniqueness}: Let the system \eqref{eqn-adjoint-regular-forward} has two solutions $\pfrak_1$ and $\pfrak_2$ with same initial data $\mfrak_0$. Then, $\wi\pfrak = \pfrak_1 - \pfrak_2$ satisfies:
\begin{align} 
	& \frac12 \frac{\d}{\d t}\|\wi\pfrak(t)\|_2^2 + \mu \|\wi\pfrak(t)\|^{2}_{\V} + \delta \langle \mathcal{C}(\pfrak_1(t)) - \mathcal{C}(\pfrak_2(t)), \wi \pfrak \rangle + \alpha \|\wi\pfrak(t)\|^{2}_{2}  
\nonumber\\ &	+  \frac{\beta}{2} \int_{\mathfrak{D}}  [|\wi\mfrak_1(x,t)|^{2}+|\wi\mfrak_2(x,t)|^{2}]|\wi\pfrak(x,t)|^2    \d x
	   + \frac{\beta}{2} \int_{\mathfrak{D}} [\{\wi\mfrak_1(x,t) + \wi\mfrak_2(x,t)\} \cdot \wi \pfrak(x,t)]^2  \d x
	\nonumber\\ & = - b( \wi \pfrak(t), \wi\mfrak_2(t), \wi\pfrak(t)) + (\wi\h (t), \hat \pfrak(t)),
\end{align}
for a.e. $t\in[0,T]$. Now, using \eqref{2.23} and a calculation similar to \eqref{3.26} yields
\begin{align} 
	& \frac12 \frac{\d}{\d t}\|\wi\pfrak(t)\|_2^2 \leq 0,
\end{align}
hence the uniqueness follows.

\vskip 2mm
\noindent
\textit{Proof of \eqref{eqn-energy-estimates-delta}-\eqref{eqn-derivative-estimates-delta}:} We have investigated that there exists a unique weak solution $$\pfrak^{\delta}\in\mathrm{C}([0,T];\H)\cap \mathrm{L}^{2}(0,T;\V)\cap\mathrm{L}^{4}(0,T;\widetilde{\L}^{4})$$ with $\frac{\d \pfrak}{\d t}\in\mathrm{L}^{2}(0,T;\V')+\mathrm{L}^{\frac{4}{3}}(0,T;\widetilde{\L}^{\frac{4}{3}})$ of system \eqref{eqn-adjoint-regular-forward} satisfying \eqref{eqn-weak-sol-forward}. Therefore, there exists a unique weak solution $\qfrak^\delta$ of system \eqref{eqn-adjoint-regular} in the sense of Definition \ref{def-adjoint-reg}. By taking into account \eqref{eqn-eeq} and the fact that 
\begin{align*}
	\|\pfrak^{\delta}\|_{\Lrm^{\infty}(0,T;\H)} &= \|\qfrak^\delta\|_{\Lrm^{\infty}(0,T;\H)}, \;\;\;  \;  \|\pfrak^{\delta}\|_{\Lrm^{2}(0,T;\V)}  = \|\qfrak^\delta\|_{\Lrm^{2}(0,T;\V)}, \;\;\;\; \|\pfrak^{\delta}\|_{\Lrm^{4}(0,T;\wi\L^4)}  = \|\qfrak^\delta\|_{\Lrm^{4}(0,T;\wi\L^4)} \\ \||\wi\mfrak_1|\pfrak^{\delta}\|_{\Lrm^{2}(0,T;\L^2(\mathfrak{D}))} & = \||\mfrak_1|\qfrak^{\delta}\|_{\Lrm^{2}(0,T;\L^2(\mathfrak{D}))},  \;\;\; \||\wi\mfrak_2|\pfrak^{\delta}\|_{\Lrm^{2}(0,T;\L^2(\mathfrak{D}))}  = \||\mfrak_2|\qfrak^{\delta}\|_{\Lrm^{2}(0,T;\L^2(\mathfrak{D}))},
\end{align*}

we achieve \eqref{eqn-energy-estimates-delta}. Let us now prove \eqref{eqn-derivative-estimates-delta}. For $\boldphi \in\mathrm{L}^{2}(0,T;\V)\cap\mathrm{L}^{4}(0,T;\widetilde{\L}^{4})$, we have 
\begin{align}\label{eqn-derivative-delta}
	& \left| \int_{0}^{T} \left\langle \frac{\d \pfrak(t)}{\d t} , \boldphi (t) \right\rangle \d t \right| 
	\nonumber\\ & \leq  \int_{0}^{T} \bigg[\mu\|\pfrak(t)\|_{\V}  + \alpha \|\pfrak(t)\|_{2} + \||\wi\mfrak_1(t)|\pfrak(t)\|_2 + \||\wi\mfrak_2(t)|\pfrak(t)\|_2 + \|\wi\h(t)\|_2 \bigg]\|\boldphi(t)\|_{\V} \d t
    \nonumber\\ & \quad  + \delta \int_{0}^{T}\|\pfrak(t)\|^3_{4} \|\boldphi(t)\|_4 \d t
	+ \frac{3\beta}{2} \int_{0}^{T} \bigg[\||\wi\mfrak_1(t)|\pfrak(t)\|_2 \|\wi\mfrak_1(t)\|_4 + \||\wi\mfrak_2(t)|\pfrak(t)\|_2 \|\wi\mfrak_2(t)\|_4 \bigg]\|\boldphi(t)\|_4 \d t
\nonumber\\ & \leq   \bigg[\mu\|\pfrak\|_{\Lrm^2(0,T;\V)}  + \alpha \|\pfrak\|_{\Lrm^2(0,T;\H)} + \||\wi\mfrak_1|\pfrak\|_{\Lrm^2(0,T;\L^2(\mathfrak{D}))} + \||\wi\mfrak_2|\pfrak\|_{\Lrm^2(0,T;\L^2(\mathfrak{D}))} + \|\wi\h\|_{\Lrm^2(0,T;\H)}   
\nonumber\\ & \quad +  \delta \|\pfrak\|^3_{\Lrm^4(0,T;\L^4(\mathfrak{D}))}  + \frac{3\beta}{2}  \||\wi\mfrak_1|\pfrak\|_{\Lrm^2(0,T;\L^2(\mathfrak{D}))} \|\wi\mfrak_1\|_{\Lrm^4(0,T;\L^4(\mathfrak{D}))} 
\nonumber\\ & \quad + \frac{3\beta}{2}  \||\wi\mfrak_2|\pfrak\|_{\Lrm^2(0,T;\L^2(\mathfrak{D}))} \|\wi\mfrak_2\|_{\Lrm^4(0,T;\L^4(\mathfrak{D}))} \bigg] \|\boldphi\|_{\Lrm^2(0,T;\V)\cap\Lrm^4(0,T;\L^4(\mathfrak{D}))}
\nonumber\\ & \leq   \bigg[\sqrt{\frac{\mu \wi K_{\h,T}}{2(1-\kappa)}}  + \sqrt{\frac{\alpha \wi K_{\h,T}}{2}} + 2 \sqrt{\frac{\wi K_{\h,T}}{(\beta-\frac{1}{2\mu\kappa})}}  + \|\wi\h\|_{\Lrm^2(0,T;\H)}  +  \delta^{\frac14} \left(\frac{\wi K_{\h,T}}{2}\right)^{\frac34}
\nonumber\\ & \quad +  \frac{3\beta}{2} \left(\frac{K_{T, \f_1,\mfrak_0}}{2\beta}\right)^{\frac{1}{4}} \sqrt{\frac{\wi K_{\h,T}}{(\beta-\frac{1}{2\mu\kappa})}}  
 +  \frac{3\beta}{2} \left(\frac{K_{T, \f_2,\mfrak_0}}{2\beta}\right)^{\frac{1}{4}} \sqrt{\frac{\wi K_{\h,T}}{(\beta-\frac{1}{2\mu\kappa})}}  
 \bigg]\|\boldphi\|_{\Lrm^2(0,T;\V)\cap\Lrm^4(0,T;\L^4(\mathfrak{D}))} \nonumber\\ & =: \left[\widehat{K}_{\h, \mu,\alpha,\beta,\kappa,T} + \delta^{\frac14} \left(\frac{\wi K_{\h,T}}{2}\right)^{\frac34}\right] \|\boldphi\|_{\Lrm^2(0,T;\V)\cap\Lrm^4(0,T;\L^4(\mathfrak{D}))}, 
\end{align}
where we have used \eqref{eqn-sol-estimate} and \eqref{eqn-energy-estimates-reqular} in the final inequality. Making use of the fact that 
\begin{align}
	\left\|\frac{\d \pfrak^{\delta}}{\d t}\right\|_{\Lrm^{2}(0,T;\V^{\prime}) + \Lrm^{\frac43}(0,T;\wi\L^{\frac43})}  = \left\|\frac{\d \qfrak^\delta}{\d t} \right\|_{\Lrm^{2}(0,T;\V^{\prime})+\Lrm^{\frac43}(0,T;\wi\L^\frac43)}
\end{align}
we complete the proof.

\vskip 2mm
\noindent
\textit{Proof of \eqref{eqn-opt-con-delta}}: Remember that $\vfrak = \mfrak_1 - \mfrak_2$ is the solution of system \eqref{eqn-difference}. Taking the inner product of equation $\eqref{eqn-difference}_1$ with $\qfrak^{\delta}$, we have
\begin{align}\label{350}
& \int_{0}^{T} \bigg[ \left\langle\frac{\d\vfrak (t)}{\d t}, \qfrak^\delta(t) \right\rangle +\mu(\nabla\vfrak(t), \nabla\qfrak^\delta(t)) + b(\mfrak_1(t) , \vfrak(t), \qfrak^\delta(t)) + b(\vfrak(t), \mfrak_2(t), \qfrak^\delta(t)) 
\nonumber\\ & + \alpha(\vfrak(t), \qfrak^\delta(t))   
   + \frac{\beta}{2}  \langle [|\mfrak_1(t)|^2 + |\mfrak_2(t)|^2] \vfrak(t)  +   [(\mfrak_1(t) + \mfrak_2(t) )\cdot\vfrak(t)] (\mfrak_1(t)+\mfrak_2(t)), \qfrak^\delta(t)\rangle \bigg] \d t
	\nonumber\\ & = \int_{0}^{T}  (\f_1(t) - \f_2(t) , \qfrak^\delta(t)) \d t.
\end{align}
Also, taking the inner  product of equation $\eqref{eqn-adjoint-regular}_1$ with $\vfrak$, we have
\begin{align}\label{351}
	& \int_{0}^{T} \bigg[ - \left\langle\frac{\d\qfrak^\delta (t)}{\d t}, \vfrak(t) \right\rangle +\mu(\nabla\vfrak(t), \nabla\qfrak^\delta(t)) + \delta \langle \mathcal{C}(\qfrak^\delta(t)) , \vfrak(t) \rangle + b(\mfrak_1(t) , \vfrak(t), \qfrak^\delta(t))
    \nonumber\\ & + b(\vfrak(t), \mfrak_2(t), \qfrak^\delta(t))  
	  + \alpha(\vfrak(t), \qfrak^\delta(t))   
	  + \frac{\beta}{2}  \langle [|\mfrak_1(t)|^2 + |\mfrak_2(t)|^2] \vfrak(t)   
      \nonumber\\ & +   [(\mfrak_1(t) + \mfrak_2(t) )\cdot\vfrak(t)] (\mfrak_1(t)+\mfrak_2(t)), \qfrak^\delta(t)\rangle \bigg] \d t
	\nonumber\\ & = \int_{0}^{T}  (\h(t) , \vfrak(t)) \d t.
\end{align}
A combination of \eqref{350} and \eqref{351} yields
\begin{align}\label{eqn-Duality-delta}
	& \int_{0}^{T}   \left\langle\frac{\d\vfrak (t)}{\d t}, \qfrak^\delta(t) \right\rangle \d t - \int_{0}^{T} (\f_1(t) - \f_2(t) , \qfrak^\delta(t)) \d t
	\nonumber\\ & = - \int_{0}^{T}    \left\langle\frac{\d\qfrak^\delta (t)}{\d t}, \vfrak(t) \right\rangle \d t - \int_{0}^{T}  (\h(t) , \vfrak(t)) \d t  + \delta  \int_{0}^{T}   \langle \mathcal{C}(\qfrak^\delta(t)) , \vfrak(t) \rangle \d t.
\end{align}

On the other hand, we have
\begin{align}\label{eqn-IbP}
	 \int_{0}^{T}   \left\langle\frac{\d\vfrak (t)}{\d t}, \qfrak^\delta(t) \right\rangle \d t 
	& = - \int_{0}^{T}    \left\langle\frac{\d\qfrak^\delta (t)}{\d t}, \vfrak(t) \right\rangle \d t + (\vfrak(T), \qfrak^\delta(T)) -  (\vfrak(0), \qfrak^\delta(0))
	 \nonumber\\ & = - \int_{0}^{T}    \left\langle\frac{\d\qfrak^\delta (t)}{\d t}, \vfrak(t) \right\rangle \d t .
\end{align}
In view of \eqref{eqn-Duality-delta} and \eqref{eqn-IbP}, we complete the proof.
\end{proof}

\subsection{Passing $\delta$ to zero} In the next proposition, we prove the existence of a weak solution for system \eqref{eqn-adjoint} in the sense of Definition \ref{def-adjoint}, derive some related estimates, and establish a connection between this solution and difference $\mfrak_1-\mfrak_2$.

\begin{proposition}\label{prop-adjoint-diff}
Suppose that the Hypothesis \ref{Para-Hypo} is satisfied. Then, there exists at least one weak solution $\qfrak$ of system \eqref{eqn-adjoint} in the sense of Definition \ref{def-adjoint}. Additionally, we have
	\begin{align}\label{eqn-Duality}
		  \int_{0}^{T} (\f_1(t) - \f_2(t) , \qfrak(t)) \d t
	 =  \int_{0}^{T}  (\h(t) , \vfrak(t)) \d t =   \int_{0}^{T}  (\h(t) , \mfrak_1(t)-\mfrak_2(t)) \d t,
	\end{align}
and the following estimates hold:
\begin{align}\label{eqn-energy-estimates}
	&   \sup_{t\in[0,T]}\|\qfrak(t)\|_2^2 + 2\mu(1-\kappa)\int_0^T \|\qfrak(t)\|^{2}_{\V} \d t 
	 	+  \left(\beta-\frac{1}{2\mu\kappa}\right) \int_0^T [ \||\mfrak_1(t)| \qfrak(t)\|_2^2 +  \|| \mfrak_1(t)| \qfrak(t)\|_2^2]   \d t
	\nonumber\\ & \leq  \wi{K}_{\h,T}, 
\end{align} 
and   
\begin{align}\label{eqn-derivative-estimates}
	\left\|\frac{\d \qfrak}{\d t} \right\|_{\Lrm^{2}(0,T;\V^{\prime})+\Lrm^{\frac43}(0,T;\wi\L^\frac43)} \leq \widehat{K}_{\h, \mu,\alpha,\beta,\kappa,T},
\end{align}
where $\wi{K}_{\h,T}$ and $\widehat{K}_{\h, \mu,\alpha,\beta,\kappa,T}$ are same as in Theorem \ref{thm-sol-delta}.
\end{proposition}
\begin{proof}
	For $\delta>0$, consider $\qfrak^\delta$ as the solution of \eqref{eqn-adjoint-regular} corresponding to given data $(\h, \mfrak_1, \mfrak_2)$. Making use of \eqref{eqn-energy-estimates-delta} and \eqref{eqn-derivative-estimates-delta}, we obtain that the sequences $\{\qfrak^\delta\}_{\delta>0}$, $\{|\mfrak_1|\qfrak^\delta\}_{\delta>0}$ and $\{|\mfrak_2|\qfrak^\delta\}_{\delta>0}$, and $\{\frac{\d \qfrak^\delta}{\d t}\}_{\delta>0}$ are uniformly bounded in $\Lrm^{\infty}(0,T;\H)\cap \Lrm^{2}(0,T;\V)$, $\Lrm^{2}(0,T;\L^2(\mathfrak{D}))$ and $\Lrm^{\frac{4}{3}}(0,T;\V^{\prime})$, respectively. In addition, we also have $\{\delta^{\frac34} [|\qfrak^\delta|^2\qfrak^\delta]\}_{\delta>0}$ is uniformly bounded in $\Lrm^{\frac43}(0,T;\L^{\frac43}(\mathfrak{D}))$. Applying the similar reasoning as in the proof of Theorem \ref{thm-sol-delta}, we deduce that there exists a sequence $\{\delta_k\}_{k\in\N}$ converging to $0$ as $k\to\infty$, and $\qfrak$ such that 
	\begin{equation}\label{eqn-convergence-delta}
	\left\{	\begin{aligned}
			\qfrak^{\delta_k}\xrightharpoonup{w^*}&\ \qfrak  &&\text{ in }	\mathrm{L}^{\infty}(0,T;\H),\\
			\qfrak^{\delta_k}\xrightharpoonup{w}&\ \qfrak  && \text{ in } \mathrm{L}^{2}(0,T;\V)\cap\mathrm{L}^{s}(0,T;\H), \;\; \text{ for all } \; s>1,  \\
			\qfrak^{\delta_k} \to &\ \qfrak  &&\text{ in }	\mathrm{L}^{2}(0,T;\H), \\
			\frac{\d \qfrak^{\delta_k}}{\d t}\xrightharpoonup{w}&\frac{\d \qfrak}{\d t}  && \text{ in } \Lrm^{2}(0,T;\V^{\prime})+\Lrm^{\frac43}(0,T;\wi\L^\frac43),\\
			|\mfrak_1|\qfrak^{\delta_k}  \xrightharpoonup{w}&\ |\mfrak_1|\qfrak  &&   \text{ in }   \;\;\; \Lrm^{2}(0,T;\L^{2}(\mathfrak{D})),  \\
			|\mfrak_2|\qfrak^{\delta_k}  \xrightharpoonup{w}&\ |\mfrak_2|\qfrak  &&  \text{ in }  \;\;\;  \Lrm^{2}(0,T;\L^{2}(\mathfrak{D})),  \\
			(  \mfrak_1 +  \mfrak_2)\cdot\qfrak^{\delta_k}  \xrightharpoonup{w}&\ (  \mfrak_1 +  \mfrak_2)\cdot\qfrak  &&  \text{ in } \;\;\; \Lrm^{2}(0,T;\L^{2}(\mathfrak{D})),  
		\end{aligned}\right.
	\end{equation}
and for every $\boldphi\in \Lrm^4(0,T;\wi\L^4)$, we have
\begin{align*}
\delta_k \int_{0}^{T} \left\langle |\qfrak^{\delta_k}(t)|^2\qfrak^{\delta_k}(t) , \boldphi(t) \right\rangle \d t & \leq \delta_k^{\frac14}\|	\delta_k^{\frac34} [|\qfrak^{\delta_k}|^2\qfrak^{\delta_k}]\|_{\Lrm^{\frac43}(0,T;\L^{\frac43}(\mathfrak{D}))} \|\boldphi\|_{\Lrm^{4}(0,T;\wi\L^{4})} \\
& \to 0 \text{ as } k\to \infty.
\end{align*}
Multiplying the weak formulation   \eqref{eqn-adjoint-reg-WF} by $g\in C_c^{\infty}((0,T])$, integration over $(0,T)$, using the convergences in \eqref{eqn-convergence-delta} and passing $k\to\infty$, we deduce that $\qfrak$ satisfies:
\begin{align*}
	&\int_{0}^{T}\left\langle - \frac{\d\qfrak(t)}{\d t} + \mu \A \qfrak(t)  \right.   -  \B(\mfrak_1(t), \qfrak(t))  +  \mathcal{P}[ \sum_{j=1}^3 [\nabla( \mfrak_2(t))_j](\qfrak(t))_j ]  + \alpha\qfrak(t)  
	\\ & \left. + \frac{\beta}{2}  \mathcal{P}\{ [| \mfrak_1(t)|^2 + | \mfrak_2(t)|^2] \qfrak(t) \}      +  \frac{\beta}{2}  \mathcal{P}\{ [(  \mfrak_1(t) +  \mfrak_2(t))\cdot\qfrak(t) ]( \mfrak_1(t) +   \mfrak_2(t))\}, \v  \right\rangle  g(t) \d t    
	\\ & = \int_{0}^{T}(   \h(t), \v) g(t)  \d t, 
\end{align*}
	for all $g\in C_c^{\infty}((0,T])$ and $\v\in \V$. This shows that $\qfrak$ is weak solution of system \eqref{eqn-adjoint} in the sense of Definition \ref{def-adjoint}. In addition, identity \eqref{eqn-Duality} holds by passing to the limit in \eqref{eqn-Duality-delta}. Finally, one can prove that $\qfrak(T)=\boldsymbol{0}$ by following the same arguments as in the Step 2 of the proof of \cite[Proposition 4.5]{Arada_2014}, we are omitting the proof here.
\end{proof}

\section{Existence of an optimal pair and first-order necessary optimality conditions}\label{sec5}\setcounter{equation}{0}

This section is devoted to the main result of this work which we give in the following two theorems.

\begin{theorem}[{Existence of optimal pair }]
	Let the Hypothesis \ref{Para-Hypo} be satisfied.	Then, there exists an optimal pair $(\wi\mfrak, \wi\f)$ of the control problem \eqref{eqn-control-problem}, where $\wi\mfrak$ is the solution of system \eqref{eqn-CBF-projected} with $\f=\wi\f$.
\end{theorem}
\begin{proof}
	The proof of this theorem closely follows from \cite[Theorem 3.3]{MTM_EECT_2022}. Therefore, we omit the details here.
\end{proof}

\begin{theorem}[First-order necessary optimality conditions]\label{thm-Main-result}
Let the Hypothesis \ref{Para-Hypo} be satisfied. Suppose that $\wi\f$ is an optimal control and $\wi\mfrak$ is the corresponding state. Then, there exists a weak solution $\wi\qfrak $ (in the sense of Definition \ref{def-main-adjoint} below)  of the following system:
	\begin{equation}\label{eqn-adjoint-NOC}
		\left\{
		\begin{aligned}
			-\frac{\d\wi\qfrak}{\d t} + \mu \A \wi\qfrak  - \B(\wi\mfrak , \wi\qfrak) + \mathcal{P}[ \sum_{j=1}^3 [\nabla (\wi\mfrak)_j]\wi\qfrak_j ]  + \alpha\wi\qfrak       +  \beta  \mathcal{P}\{ |\wi\mfrak|^2 \wi\qfrak \}     +  2 \beta \mathcal{P}\{ [\wi\mfrak \cdot\wi\qfrak ] \wi\mfrak \}  & = \mathcal{P} (\wi \mfrak - \mfrak_d), \\
			\wi\qfrak(T)&=\boldsymbol{0}.
		\end{aligned}
		\right.
	\end{equation} 
such that 
\begin{align}
	\int_{0}^{T} (\v(t) - \wi\f(t)   ,  \wi\qfrak (t) + \lambda \wi\f(t))\d t \geq 0, \;\; \; \text{ for all }\;\; \v \in \mathcal{F}_{ad}
\end{align}
\end{theorem}

Let us now provide the definition of weak solution to system \eqref{eqn-adjoint-NOC}.
\begin{definition}\label{def-main-adjoint}
Let $\wi\mfrak$ be the solution of system \eqref{eqn-CBF-projected} with $\f=\wi\f$.	A function  $\qfrak\in\mathrm{L}^{\infty}(0,T;\H)\cap\mathrm{L}^2(0,T;\V)$ with  
	\begin{align*}
	\partial_t\qfrak\in\Lrm^{2}(0,T;\V^{\prime})+\Lrm^{\frac43}(0,T;\wi\L^\frac43),\;\; 	\int_{0}^{T}	\int_{\mathfrak{D}} |\wi\mfrak(x,t)|^{2}|\wi\qfrak(x, t)|^2   \d x \d t   < + \infty, \;\;\;  \mbox{and $\wi\qfrak(0)= \boldsymbol{0}$,}
	\end{align*}
	is called a \emph{weak solution} to the system \eqref{eqn-adjoint-NOC}, if for $  \mfrak_d \in\mathrm{L}^2(0,T;\H)$ and $\v\in\V$, $\qfrak(\cdot)$ satisfies:
	\begin{equation} 	 
	\begin{aligned}
			& \bigg\langle-\frac{\d\wi\qfrak}{\d t} + \mu \A \wi\qfrak - \B(\wi\mfrak, \qfrak)  + \mathcal{P}[ \sum_{j=1}^3 [\nabla(\wi\mfrak)_j]\qfrak_j ]   + \alpha\qfrak      +  \beta  \mathcal{P}\{ |\wi\mfrak|^2  \wi\qfrak \}       + 2\beta \mathcal{P}\{ [\wi\mfrak\cdot\wi\qfrak]\wi\mfrak\} ,\v\bigg\rangle \\ &=(\wi\mfrak - \mfrak_d, \v),
		\end{aligned}
		 	\end{equation}
 for a.e.  $t\in[0,T]$.
\end{definition}

In order to establish the first-order necessary optimality conditions (or proving Theorem \ref{thm-Main-result}), we provide intermediate optimality conditions which constitute key steps
 in the subsequent analysis. 

Assume that $(\wi\mfrak, \wi\f)$ is a solution of control problem \eqref{eqn-control-problem} and let $\u\in \mathcal{F}_{ad}$. For $0<\rho<1$, suppose that $\wi\f_{\rho} = \wi\f + \rho (\u - \wi\f)$ and $\wi\mfrak_\rho$ is the solution of system \eqref{eqn-CBF-projected} corresponding to $\wi\f_\rho$.

\begin{theorem}
Let the Hypothesis \ref{Para-Hypo} be satisfied. There exists a weak solution $\wi\qfrak_\rho \in\mathrm{L}^{\infty}(0,T;\H)\cap\mathrm{L}^2(0,T;\V)$ with  $\qfrak(T)=\boldsymbol{0}$,
\begin{align*}
	\partial_t\wi\qfrak_\rho\in \mathrm{L}^{2}(0,T;\mathbb{V}')+\mathrm{L}^{\frac{4}{3}}(0,T;\wi\L^{\frac{4}{3}}), \text{ and } \;\;\;	\int_{0}^{T}	\int_{\mathfrak{D}}  [|\wi\mfrak(x,t)|^{2}+|\wi\mfrak_\rho(x,t)|^{2}]|\wi\qfrak_\rho(x, t)|^2   \d x \d t   < + \infty,
\end{align*}
 to the following system: 
 \begin{equation}\label{eqn-adjoint-diff-rho}
 	\begin{aligned}
 		-\frac{\d\qfrak}{\d t} + \mu \A \qfrak - \B(\wi\mfrak, \qfrak) + \mathcal{P}[ \sum_{j=1}^3 [\nabla(\wi\mfrak_\rho)_j]\qfrak_j ]  + \alpha\qfrak  & + \frac{\beta}{2} \mathcal{P}\{ [|\wi\mfrak|^2 + |\wi\mfrak_\rho|^2] \qfrak \} \\  + \frac{\beta}{2}\mathcal{P}\{ [(\wi\mfrak + \wi\mfrak_\rho)\cdot\qfrak](\wi\mfrak+\wi\mfrak_\rho)\}  & = \mathcal{P}(\wi\mfrak - \mfrak_d), 
 	\end{aligned}
 \end{equation}
 such that 
 \begin{align}\label{eqn-IOC}
 	\int_{0}^{T} (\u(t) - \wi\f(t)   ,  \wi\qfrak_\rho (t) + \lambda \wi\f(t))\d t + \frac{\rho}{2} \int_{0}^{T} \left\|\frac{\wi\mfrak_\rho(t) - \wi\mfrak(t)}{\rho}\right\|^2_2\d t + \frac{\rho\lambda}{2} \int_{0}^{T}  \|\u(t) - \wi\f (t) \|^2_2\d t \geq 0, 
 \end{align}
 for all  $\u \in  \mathcal{F}_{ad}$. Moreover, $\wi\qfrak_\rho$ satisfies the following estimates:
 \begin{equation}
 \label{eqn-adjoint-energy-rho}
 	\left\{
 	\begin{aligned}
 		\|\wi\qfrak_\rho\|^2_{\Lrm^{\infty}(0,T;\H)}   \leq \wi {K}_{\wi\mfrak - \mfrak_d,T} , \;\;\;    \|\wi\qfrak_\rho\|^2_{\Lrm^{2}(0,T;\V)} & \leq \frac{\wi {K}_{\wi\mfrak - \mfrak_d ,T}}{2\mu(1-\kappa)}  , 
 		\\ \||\wi\mfrak|\wi\qfrak_\rho\|^2_{\Lrm^{2}(0,T;\L^2(\mathfrak{D}))}  +  \||\wi\mfrak_\rho|\wi\qfrak_\rho\|^2_{\Lrm^{2}(0,T;\L^2(\mathfrak{D}))} & \leq \frac{\wi {K}_{\wi\mfrak - \mfrak_d ,T}}{\beta-\frac{1}{2\mu\kappa}},  
 \\ 	\left\|\frac{\d \wi\qfrak_\rho}{\d t} \right\|_{\mathrm{L}^{2}(0,T;\mathbb{V}')+\mathrm{L}^{\frac{4}{3}}(0,T;\wi\L^{\frac{4}{3}})} & \leq \widehat{K}_{\wi\mfrak - \mfrak_d , \mu,\alpha,\beta,\kappa,T},
 	\end{aligned}
 \right.
 \end{equation}
where $\wi{K}$ and $\widehat{K}$ are  defined  similarly as in Proposition \ref{prop-adjoint-diff}.
\end{theorem}
\begin{proof}
	By the definition of cost functional  $\mathcal{J}$, we have
	\begin{align}\label{eqn-cost-variation}
		\frac{\mathcal{J}(\wi\mfrak_\rho , \wi\f_\rho) - \mathcal{J}(\wi\mfrak , \wi\f)}{\rho} 
		  & = \int_{0}^{T} [(\frac{\wi\mfrak_\rho(t) - \wi\mfrak(t)}{\rho} , \wi\mfrak(t) - \mfrak_d(t))  + (\lambda\wi\f(t) , \u(t) - \wi\f(t))] \d t
		  \nonumber\\ & \quad +
		   \frac{\rho}{2} \int_{0}^{T} \left\|\frac{\wi\mfrak_\rho(t) - \wi\mfrak(t)}{\rho}\right\|^2_2\d t + \frac{\rho\lambda}{2} \int_{0}^{T}  \|\u(t) - \wi\f (t) \|^2_2\d t \geq 0.
	\end{align}
In view of Proposition \ref{prop-adjoint-diff}, system \eqref{eqn-adjoint-diff-rho} has a weak solution $\wi\qfrak_\rho\in\mathrm{L}^{\infty}(0,T;\H)\cap\mathrm{L}^2(0,T;\V)$ with  $\wi\qfrak_\rho(0) = \boldsymbol{0}$,
\begin{align*}
\partial_t\wi\qfrak_\rho\in\mathrm{L}^{2}(0,T;\mathbb{V}')+\mathrm{L}^{\frac{4}{3}}(0,T;\wi\L^{\frac{4}{3}})  \; \text{ and }\;	\int_{0}^{T}	\int_{\mathfrak{D}}  [|\wi\mfrak(x,t)|^{2}+|\wi\mfrak_\rho(x,t)|^{2}]|\wi\qfrak_\rho(x, t)|^2   \d x \d t  < + \infty,
\end{align*}
which satisfy
	\begin{align}\label{eqn-Duality-rho}
	\int_{0}^{T} ( \u(t) - \wi\f(t) , \wi\qfrak_\rho(t)) \d t
	=  \int_{0}^{T}  (\wi\mfrak(t) - \mfrak_d(t) , \frac{\wi\mfrak_\rho(t) - \wi\mfrak(t)}{\rho}) \d t,
\end{align}
and all the estimates in \eqref{eqn-adjoint-energy-rho}.  Finally, a combination of \eqref{eqn-cost-variation} and \eqref{eqn-Duality-rho} gives \eqref{eqn-IOC}.
\end{proof}

\subsection{Proof of Theorem \ref{thm-Main-result}}
Proposition \ref{prop-stability-diff} tells that the following estimates hold:
\begin{align}\label{eqn-rho-diff}
	 &\sup_{0\leq t \leq T} \|\wi\mfrak_\rho(t) - \wi\mfrak(t)\|_2^2 + 2 \mu(1-\kappa) \int_{0}^{T} \|\wi\mfrak_\rho(t) - \wi\mfrak(t)\|^2_{\V} \d t    +  \frac{1}{2}\left(\beta - \frac{1}{2\mu\kappa}\right) \int_{0}^{T} \|\wi\mfrak_\rho(t) - \wi\mfrak(t) \|_{4}^4 \d t  \nonumber\\ & \leq \rho^2 e^T \int_{0}^{T}\|\v(t) -\wi\f(t)\|^2_{2}\d t.
\end{align}
We infer from the above estimates that 
\begin{align}\label{eqn-strong-rho}
	\mbox{$\wi\mfrak_\rho$ strongly converges to $\wi\mfrak$ in $\mathrm{L}^{\infty}(0,T;\H)\cap\mathrm{L}^2(0,T;\V)\cap \mathrm{L}^4(0,T;\wi\L^4)$ as $\rho\to 0$.}
\end{align}
 In addition, due to \eqref{eqn-adjoint-energy-rho}, there exists a subsequence $\{\rho_k\}_{k\in\N}$ converging to $0$ as $k\to\infty$, such that 
 	\begin{equation}\label{eqn-convergence-rho}
 	\left\{	\begin{aligned}
 		\wi\qfrak_{\rho_k}\xrightharpoonup{w}&\ \wi\qfrak  && \text{ in } \mathrm{L}^{2}(0,T;\V)\cap\mathrm{L}^{s}(0,T;\H), \;\; \text{ for all } \; s>1,  \\
 		\wi\qfrak_{\rho_k} \to &\ \wi\qfrak  &&\text{ in }	\mathrm{L}^{2}(0,T;\H), \\
 		\frac{\d \wi\qfrak_{\rho_k}}{\d t}\xrightharpoonup{w}&\frac{\d \wi\qfrak}{\d t}  && \text{ in }\mathrm{L}^{\frac{4}{3}}(0,T;\V'),\\
 		|\wi\mfrak|\wi\qfrak_{\rho_k}  \xrightharpoonup{w}&\ |\wi\mfrak|\wi\qfrak  &&   \text{ in }   \;\;\; \Lrm^{2}(0,T;\L^{2}(\mathfrak{D})), \\
 		\wi\mfrak \cdot \wi\qfrak_{\rho_k}  \xrightharpoonup{w}&\ \wi\mfrak \cdot \wi\qfrak  &&   \text{ in }   \;\;\; \Lrm^{2}(0,T;\L^{2}(\mathfrak{D})). 
 	\end{aligned}\right.
 \end{equation}
  Note that $\wi\qfrak_{\rho_k}$ satisfies the following weak formulation:
\begin{align}\label{eqn-49}
&\int_{0}^{T}	\bigg\langle -\frac{\d\wi\qfrak_{\rho_k}(t)}{\d t} + \mu \A \wi\qfrak_{\rho_k}(t) - \B(\wi\mfrak(t), \wi\qfrak_{\rho_k}(t))  +  [ \sum_{j=1}^3 [\nabla(\wi\mfrak_{\rho_k}(t))_j](\wi\qfrak_{\rho_k}(t))_j ]   + \alpha\wi\qfrak_{\rho_k}(t)   
\nonumber\\ & + \frac{\beta}{2}  \{ [|\wi\mfrak(t)|^2 + |\wi\mfrak_{\rho_k}(t)|^2] \wi\qfrak_{\rho_k}(t) \}      + \frac{\beta}{2}  \{ [(\wi\mfrak (t) + \wi\mfrak_{\rho_k}(t))\cdot\wi\qfrak_{\rho_k}(t)](\wi\mfrak(t)+\wi\mfrak_{\rho_k}(t))\}, \boldphi\bigg\rangle g(t) \d t 
\nonumber\\ & = \int_{0}^{T} (\wi\mfrak(t) - \mfrak_d(t), \boldphi) g(t) \d t ,
\end{align}
for all $\boldphi \in \V$ and $g\in C_c^{\infty}((0,+\infty))$.
  Using the strong and weak convergences in \eqref{eqn-convergence-rho}, we aim to pass to the limit as $k\to\infty$ in \eqref{eqn-49}.
\vskip 2mm
\noindent
\textit{Step 1}: Applying the convergences provided in \eqref{eqn-convergence-rho}, we get
\begin{align}
\lim_{k\to\infty}	\int_{0}^{T}	\bigg\langle -\frac{\d\wi\qfrak_{\rho_k}(t)}{\d t} , \boldphi\bigg\rangle g(t) \d t  
& = \int_{0}^{T}	\bigg\langle -\frac{\d\wi\qfrak(t)}{\d t} , \boldphi\bigg\rangle g(t) \d t, \\
	\lim_{k\to\infty}	\int_{0}^{T}	\langle \mu\A\wi\qfrak_{\rho_k}(t) + \alpha \wi\qfrak_{\rho_k} , \boldphi \rangle g(t) \d t & = \int_{0}^{T}	\langle \mu\A\wi\qfrak (t) + \alpha \wi\qfrak  , \boldphi \rangle g(t) \d t, \\
	\lim_{k\to\infty}	\int_{0}^{T}	\langle      |\wi\mfrak(t)|^2  \wi\qfrak_{\rho_k}(t)  , \boldphi \rangle g(t) \d t &  = \lim_{k\to\infty}	\int_{0}^{T}	\langle    |\wi\mfrak(t)|  \wi\qfrak_{\rho_k}(t)   , |\wi\mfrak(t)| \boldphi \rangle g(t) \d t
	\nonumber\\ & = \int_{0}^{T}	\langle    |\wi\mfrak(t)|  \wi\qfrak (t)   , |\wi\mfrak(t)| \boldphi \rangle g(t) \d t
    \nonumber \\ & = \int_{0}^{T}	\langle      \{ |\wi\mfrak(t)|^2  \wi\qfrak (t) \}  , \boldphi \rangle g(t) \d t .\label{eqn-414}
\end{align}

Now, we consider
\begin{align}
&	\int_{0}^{T}	\langle      |\wi\mfrak_{\rho_k}(t)|^2  \wi\qfrak_{\rho_k}(t)   , \boldphi \rangle g(t) \d t 
\nonumber\\ & =  	\underbrace{\int_{0}^{T}	\langle      |\wi\mfrak (t)|^2 \wi\qfrak_{\rho_k}(t)   , \boldphi \rangle g(t) \d t}_{I_{1,k}}
 +	\underbrace{\int_{0}^{T}	\langle      [(\wi\mfrak_{\rho_k}(t)-\wi\mfrak(t)) \cdot  \wi\mfrak(t)] \wi\qfrak_{\rho_k}(t)   , \boldphi \rangle g(t) \d t}_{I_{2,k}}
\nonumber\\ & \quad   +  \underbrace{\int_{0}^{T}	\langle      [\wi\mfrak_{\rho_k}(t)\cdot  (\wi\mfrak_{\rho_k}(t)- \wi\mfrak(t))] \wi\qfrak_{\rho_k}(t)   , \boldphi \rangle g(t) \d t}_{I_{3,k}}
\end{align}
Due to strong convergence in \eqref{eqn-strong-rho} and uniform estimate in \eqref{eqn-adjoint-energy-rho}, it is immediate that $\lim\limits_{k\to \infty}I_{2,k} = \lim\limits_{k\to \infty}I_{3,k}= 0.$ Therefore, along with \eqref{eqn-414}, we obtain 
\begin{align}
	\lim\limits_{k\to \infty} \int_{0}^{T}	\langle      |\wi\mfrak_{\rho_k}(t)|^2  \wi\qfrak_{\rho_k}(t)   , \boldphi \rangle g(t) \d t 
		=  	 \int_{0}^{T}	\langle      |\wi\mfrak (t)|^2 \wi\qfrak (t)   , \boldphi \rangle g(t) \d t.
\end{align}

Since $ \langle  -\B(\wi\mfrak, \wi\qfrak_{\rho_k})  +       \sum_{j=1}^3 [\nabla(\wi\mfrak_{\rho_k})_j](\wi\qfrak_{\rho_k})_j    , \boldphi  \rangle = b( \wi\mfrak, \boldphi  , \wi\qfrak_{\rho_k}   ) +   b(\boldphi, \wi\mfrak_{\rho_k} , \wi\qfrak_{\rho_k}  )  $, we consider
\begin{align}\label{eqn-417}
 &  	   \int_{0}^{T}	[b( \wi\mfrak  (t),\boldphi, \wi\qfrak_{\rho_k}(t) ) +   b(\boldphi, \wi\mfrak_{\rho_k} (t), \wi\qfrak_{\rho_k}(t)   )] g(t) \d t - \int_{0}^{T}	[b( \wi\mfrak  (t),\boldphi, \wi\qfrak (t) ) +   b(\boldphi, \wi\mfrak  (t), \wi\qfrak (t)   )] g(t) \d t   
 \nonumber\\ & = \underbrace{\int_{0}^{T}	b( \wi\mfrak  (t),\boldphi, \wi\qfrak_{\rho_k}(t) - \wi\qfrak (t) ) g(t) \d t}_{I_{4,k}} + \underbrace{\int_{0}^{T}	b(\boldphi, \wi\mfrak_{\rho_k} (t) - \wi\mfrak  (t), \wi\qfrak_{\rho_k}(t)   ) g(t) \d t }_{I_{5,k}}
 \nonumber\\ & \quad + \underbrace{\int_{0}^{T}	b( \boldphi,\wi\mfrak  (t), \wi\qfrak_{\rho_k}(t) - \wi\qfrak(t) ) g(t) \d t}_{I_{6,k}}.
\end{align}
Taking into account the weak convergence $\eqref{eqn-convergence-rho}_1$, and the fact that $\boldphi\in\V$, $g\in C_c^{\infty}((0,+\infty))$, $\wi\mfrak\in \Lrm^4(0,T;\wi\L^4)$, it is immediate that $\lim\limits_{k\to \infty}I_{4,k}=\lim\limits_{k\to \infty}I_{6,k}= 0.$ Moreover, using the fact that $\{\wi\qfrak_{\rho_k}\}_{k\in\N}$ is uniformly bounded in $\Lrm^2(0,T;\V)$, $\boldphi\in\V$, $g\in C_c^{\infty}((0,+\infty))$ and $\wi\mfrak_{\rho_k}\to \wi\mfrak$ in $\Lrm^4(0,T;\wi\L^4)$, we obtain that $\lim\limits_{k\to \infty}I_{5,k}= 0.$ Therefore, we have from \eqref{eqn-417} that
\begin{align}
    &\lim_{k\to\infty}   \int_{0}^{T}	[b( \wi\mfrak  (t),\boldphi, \wi\qfrak_{\rho_k}(t) ) +   b(\boldphi, \wi\mfrak_{\rho_k} (t), \wi\qfrak_{\rho_k}(t)   )] g(t) \d t 
    \nonumber\\ & =  \int_{0}^{T}	[b( \wi\mfrak  (t),\boldphi, \wi\qfrak (t) ) +   b(\boldphi, \wi\mfrak  (t), \wi\qfrak (t)   )] g(t) \d t.
\end{align}
Finally, we consider
\begin{align}
&	\int_{0}^{T}	\langle   [(\wi\mfrak (t) + \wi\mfrak_{\rho_k}(t))\cdot\wi\qfrak_{\rho_k}(t)](\wi\mfrak(t)+\wi\mfrak_{\rho_k}(t)), \boldphi \rangle g(t) \d t  - 4 \int_{0}^{T}	\langle   [ \wi\mfrak (t)\cdot\wi\qfrak (t)]\wi\mfrak(t), \boldphi \rangle g(t) \d t 
	\nonumber\\ & =\underbrace{ 2 \int_{0}^{T}	\langle   \wi\mfrak (t)\cdot\wi\qfrak_{\rho_k}(t)(\wi\mfrak_{\rho_k}(t)-\wi\mfrak(t)), \boldphi \rangle g(t) \d t}_{I_{7,k}} 
    \nonumber\\ & \quad +  \underbrace{\int_{0}^{T}	\langle   [ (\wi\mfrak_{\rho_k}(t)-\wi\mfrak(t))\cdot\wi\qfrak_{\rho_k} (t)](\wi\mfrak(t)+\wi\mfrak_{\rho_k}(t)), \boldphi \rangle g(t) \d t }_{I_{8,k}}
	\nonumber\\ & \quad + \underbrace{4 \int_{0}^{T}	\langle   [ \wi\mfrak (t)\cdot(\wi\qfrak_{\rho_k} (t) - \wi\qfrak (t))]\wi\mfrak(t), \boldphi \rangle g(t) \d t }_{I_{9,k}}.\label{eqn-419}
\end{align}
 Making use of the weak convergence $\eqref{eqn-convergence-rho}_5$, and the fact that $\boldphi\in\V$, $g\in C_c^{\infty}((0,+\infty))$, $\wi\mfrak\in \Lrm^4(0,T;\wi\L^4)$, it is immediate that $\lim\limits_{k\to \infty}I_{9,k}= 0.$ Moreover, using the fact that $\{|\wi\mfrak|\wi\qfrak_{\rho_k}, \; |\wi\mfrak_{\rho_k}|\wi\qfrak_{\rho_k} \}_{k\in\N}$ is uniformly bounded in $\Lrm^2(0,T; \L^2(\mathfrak{D}))$, $\boldphi\in\V$, $g\in C_c^{\infty}((0,+\infty))$ and $\wi\mfrak_{\rho_k}\to \wi\mfrak$ in $\Lrm^4(0,T;\wi\L^4)$, we obtain that $\lim\limits_{k\to \infty}I_{7,k}= \lim\limits_{k\to \infty}I_{8,k}= 0.$ Therefore, from  \eqref{eqn-419} we have 
 \begin{align}
 	\lim_{k\to\infty}   \int_{0}^{T}	\langle   [(\wi\mfrak (t) + \wi\mfrak_{\rho_k}(t))\cdot\wi\qfrak_{\rho_k}(t)](\wi\mfrak(t)+\wi\mfrak_{\rho_k}(t)), \boldphi \rangle g(t) \d t  = 4 \int_{0}^{T}	\langle   [ \wi\mfrak (t)\cdot\wi\qfrak (t)]\wi\mfrak(t), \boldphi \rangle g(t) \d t .
 \end{align}
Hence passing to the limit to the weak formulation \eqref{eqn-49}, we infer that $\wi\qfrak$ satisfies:
\begin{align} 
	&\int_{0}^{T}	\bigg\langle -\frac{\d\wi\qfrak(t)}{\d t} + \mu \A \wi\qfrak (t) - \B(\wi\mfrak(t), \wi\qfrak (t))  +  [ \sum_{j=1}^3 [\nabla(\wi\mfrak (t))_j](\wi\qfrak (t))_j ]   + \alpha\wi\qfrak (t)   
 +  \beta   |\wi\mfrak(t)|^2  \wi\qfrak (t)   
   	\nonumber\\ & + 4  [\wi\mfrak (t) \cdot\wi\qfrak (t)] \wi\mfrak(t), \boldphi\bigg\rangle g(t) \d t 
	  = \int_{0}^{T} (\wi\mfrak(t) - \mfrak_d(t), \boldphi) g(t) \d t.
\end{align} 
We conclude that $\wi\qfrak$ is a weak solution of system \eqref{eqn-adjoint-NOC} in the sense of Definition \ref{def-main-adjoint}. Considering \eqref{eqn-rho-diff} and passing to the  limit in  \eqref{eqn-Duality-rho} as $\rho=\rho_k\to 0$, we get \eqref{eqn-Duality}, as desired. This completes the proof.

\vskip 2mm
\noindent
\textbf{Acknowledgments:} “This work is funded by national funds through the FCT – Fundação para a Ciência e a Tecnologia, I.P., under the scope of the projects UIDB/00297/2020 (https://doi.org/ 10.54499/UIDB/00297/2020) and UIDP/00297/2020 (https://doi.org/10.54499/UIDP/00297/ 2020) (Center for Mathematics and Applications)”.

\medskip\noindent
\textbf{Data availability:} No data was used for the research described in the article.

\medskip\noindent
\textbf{Declarations}: During the preparation of this work, the authors have not used AI tools.

\medskip\noindent
\textbf{Conflict of interest:} The authors declare no conflict of interest.

\centering
\bibliographystyle{abbrv}

\begin{thebibliography}{99}
	
	
	
		\bibitem{Arada_2014}   N. Arada, Optimal control of evolutionary quasi-Newtonian fluids, \emph{SIAM J. Control Optim.}, {\bf 52}(5) (2014), 3401--3436.   
	
	
	\bibitem{CV} V.V. Chepyzhov and M.I. Vishik,  \emph{Attractors for Equations of Mathematical Physics},  American Mathematical Society, Providence, Rhode Island, 2002.
	
	
	\bibitem{GPG}	  G.P. Galdi,	  An introduction to the Navier-Stokes initial-boundary value problem,	  In \emph{Fundamental Directions in Mathematical Fluid Mechanics}, Adv. Math. Fluid Mech. Birkh\"auser, Basel, 2000, 1-70.
	
	
	
	\bibitem{Gautam+Mohan_2025}   S. Gautam and M. T. Mohan, 	On convective Brinkman-Forchheimer equations, \emph{Dynamics of Partial Differential Equations}, \textbf{22}(3) (2025), pp. 191--233.
	
	\bibitem{Gautam+Mohan_Arxiv_2025}  S. Gautam and M.T. Mohan, Optimal control of convective Brinkman-Forchheimer equations: Dynamic programming equation and Viscosity solutions, \emph{arXiv preprint} \url{arXiv:2505.07095} (2025).
	
	
	
	\bibitem{Gunzburger_1995}  M.D. Gunzburger, \emph{Flow control}, Proceedings of the Workshop on Period of Concentration in Flow Control Held at the University of Minnesota, 1992, The IMA Volumes in Mathematics and Its Applications, 68, Springer-Verlag, New York, 1995.
	
	
	
	\bibitem{Hajduk+Robinson}  K.W. Hajduk and J.C. Robinson, Energy equality for the 3D critical convective Brinkman-Forchheimer
	equations, \emph{Journal of Differential Equations}, \textbf{263} (2017), 7141--7161.
	
	
	\bibitem{Hopf_1951} E. Hopf, \"Uber die Anfangswertaufgabe f\"ur die hydrodynamischen Grundgleichungen, \emph{Math. Nachr.}, \textbf{4} (1951), 213--231.
	
	
	
	\bibitem{OAL}	O. A. Ladyzhenskaya, \emph{The Mathematical Theory of Viscous Incompressible Flow}, Gordon and Breach, New York, 1969.
	
	
	
	
	\bibitem{Leray_1934}   J. Leray, Sur le mouvement d'un liquide visqueux emplissant l'espace, \emph{Acta Math.}, \textbf{63}(1) (1934), 193--248.
	
	\bibitem{Lions_1969}   J.-L. Lions, {\it Quelques m\'ethodes de r\'esolution des probl\`emes aux limites non lin\'eaires}, Dunod, Paris, 1969 Gauthier-Villars, Paris, 1969.
	
	
	
	\bibitem{MTM_SPDE_2022} M.T. Mohan,  Well-posedness and asymptotic behavior of stochastic convective Brinkman–Forchheimer equations perturbed by pure jump noise, \emph{Stoch PDE: Anal Comp}, \textbf{10} (2022), 614--690.
	

	
	
	
	\bibitem{MTM_AMOP_2021}  M.T. Mohan, The time optimal control of two dimensional convective Brinkman-Forchheimer equations, \emph{Appl. Math. Optim.}, {\bf 84}(3) (2021), 3295--3338.
	
	
	
	\bibitem{MTM_EECT_2022}  M.T. Mohan, Optimal control problems governed by two dimensional convective Brinkman-Forchheimer equations, \emph{Evol. Equ. Control Theory}, \textbf{11}(3) (2022),  649-679. 
	
	
	\bibitem{MTM_Optimization_2022} M.T. Mohan, First-order necessary conditions of optimality for the optimal control of two-dimensional convective Brinkman-Forchheimer equations with state constraints, \emph{Optimization}, {\bf 71}(13) (2022),  3861--3907.
	
	
	
	\bibitem{Sritharan_1998}  S.S. Sritharan, \emph{An introduction to deterministic and stochastic control of viscous flow}, Optimal Control of Viscous Flow, 1–42, SIAM, Philadelphia, PA, 1998.
	
	
	


\bibitem{Temam_1984}  R. Temam, Navier–Stokes Equations: Theory and Numerical Analysis, North-Holland, Amsterdam,
1984.


    


\end{thebibliography}

\end{document}